\documentclass[12pt,reqno,a4paper]{amsart}

\usepackage{soul}
\usepackage{amsthm,mathtools}
\usepackage{amssymb}
\usepackage{latexsym}
\usepackage{multicol,multirow}
\usepackage{verbatim,enumerate}
\usepackage{accents,upgreek}
\usepackage[usenames]{color}
\usepackage[colorlinks=true, linkcolor=blue, citecolor=green, urlcolor=blue]{hyperref}
\usepackage[nameinlink,noabbrev,capitalize]{cleveref}
\usepackage{nicematrix,makecell}
\usepackage{amsmath, amscd}
\usepackage{soul,enumitem}
\usepackage{pifont}
\usepackage{dsfont}
\usepackage{dirtytalk,caption,subcaption}
\usepackage{tikz,setspace}
\usetikzlibrary{matrix,arrows,decorations.pathmorphing}
\usetikzlibrary{positioning}

\advance\textwidth by 1.2in \advance\oddsidemargin by -.6in \advance\evensidemargin by -.6in

\theoremstyle{plain}
\newtheorem*{prop}{Proposition}
\newtheorem{thm}{Theorem}
\newtheorem*{theom}{Theorem}
\newtheorem*{lem}{Lemma}
\newtheorem*{cor}{Corollary}

\theoremstyle{definition}
\newtheorem*{example}{Example}

\theoremstyle{remark}
\newtheorem{rem}{Remark}

\newcommand{\lie}[1]{\mathfrak{#1}}

\newcommand\br{\mathbb R}
\newcommand\bc{\mathbb C}
\newcommand\bz{\mathbb Z}
\newcommand\bn{\mathbb N}

\newcounter{cnt}
 \makeatletter
\def\mydggeometry{\makeatletter\dg@YGRID=1\dg@XGRID=20\unitlength=0.003pt\makeatother}
\makeatother \theoremstyle{remark}

\numberwithin{equation}{section}
\makeatletter
\def\section{\def\@secnumfont{\mdseries}\@startsection{section}{1}%
	\z@{.7\linespacing\@plus\linespacing}{.5\linespacing}%
	{\normalfont\scshape\centering}}
\def\subsection{\def\@secnumfont{\bfseries}\@startsection{subsection}{2}%
	{\parindent}{.5\linespacing\@plus.7\linespacing}{-.5em}%
	{\normalfont\bfseries}}
\makeatother

\title[Polytopes and isomorphisms of KR crystals]{On the Polytope Model and Near End node Isomorphisms of Type $A$ Kirillov--Reshetikhin Crystals}
\author{Dipnit Biswas}
\address{Indian Institute of Technology Madras, Chennai, India}
\email{ma25r011@smail.iitm.ac.in}

\author{Irfan Habib}\address{The Institute of Mathematical Sciences, A CI Of Homi Bhabha National Institute, Chennai 600113, India}
\email{irfanhabib@imsc.res.in}
\thanks{}
\begin{document}
\begin{abstract}
	We prove an inductive formula to construct a path from the highest weight element to any given vertex in the crystal graph of the polytope realization of the Kirillov–-Reshetikhin crystal $KR^{i,m}$ of type $A$. For $i \leq 2$ or $i \geq n-1$, we provide explicit formulas of the same by only using the lowering crystal operators and in those cases, using these paths, we determine the explicit image of any element under the affine crystal isomorphisms between the polytope and the tableau realizations of the Kirillov–-Reshetikhin crystals.
\end{abstract}
\maketitle
\section{Introduction}
Let $\lie g$ be a finite-dimensional simple Lie algebra. Let $\hat{\lie g}$ be the associated untwisted affine Lie algebra and $U_q'(\hat{\lie g})$ be the quantum affine algebra without the degree derivation. Finite-dimensional simple modules for $U_q'(\hat{\lie g})$ have been classified in \cite{CP95QAA} in terms of their Drinfeld polynomials. The class of the prime simple objects in the category $\mathcal{C}_{\hat{\lie{g}}}$ of the finite-dimensional representations of $U_q'(\hat{\lie g})$ is still unknown. Kirillov-Reshetikhin (KR in short) modules $KR(m\omega_i)$ are prime simple objects in $\mathcal{C}_{\hat{\lie{g}}}$, indexed by a node $i$ of the finite Dynkin diagram and a positive integer $m$. In \cite{CP91CMP} Chari and Pressley proved that for ${\lie g}={\lie{sl}_2}$, then the Kirillov-Reshetikhin modules exhaust all prime simple objects in $\mathcal{C}_{\hat{\lie {g}}}$. KR modules have attracted significant attention due to their rich structure and interesting properties such as $T$-systems, fermionic character formulas; see \cite{Chari01Fermionic,Hernandez06KRconj,Hernandez10KRconj} and the references therein for more details.

One of the ways to understand the modules in the category $\mathcal{C}_{\hat{\lie g}}$ is to study them via their crystal bases, if they exist. In \cite{Hata02Paths} it was conjectured that the KR modules admit crystal bases in the sense of \cite{Kash91Oncrystal} and the conjecture was proved in \cite{Kang92Perfect} for $A_n^{(1)}$ and in \cite{Okado08Existence} for non-exceptional types. When a crystal base for $KR(m\omega_i)$ exists, we call the associated crystal the KR crystal $KR^{i,m}$. There have been researches where these crystals have been realized using explicit combinatorial models. If $\lie g$ is of type $A_n$, then $KR^{i,m}$ can be realized as the semistandard Young tableau of rectangular shape $(m^i)$ with filling from $1,2,\dots,n+1$ and the crystal operators are defined as in \cite{Kashiwara94CrystalGraph,Shimozono02AffineTypeA}. The explicit combinatorial models for all non-exceptional types are given in \cite{Fourier09AdvKirillov}.

In \cite{Feigin11PBW}, when $\lie g$ is of type $A_n,$ the authors constructed a polytope for each dominant integral weight $\lambda$ and showed that the basis elements are parametrized by the integral points of the polytope.  Kus in \cite{Kus13RealizationJCTA} defined an affine crystal structure on the integral points of the polytope when $\lambda=m\omega_i$ and proved that the resulting crystal, denoted  by him as $B^{i,m}$, is isomorphic to the affine KR crystal $KR^{i,m}$. Later, in \cite{Kus16KirillovJACO}, the same author provided a simplified formula for the action of the affine node. Similar polytopes are known for the other types as well e.g. see \cite{Feigin11PBWC} for type $C$ and \cite{Back19PBW} for type $B$.

This article is a step towards understanding the explicit isomorphism between the two combinatorial models for KR crystals, namely the tableau model in \cite{Kashiwara94CrystalGraph,Shimozono02AffineTypeA} and the polytope model given in \cite{Kus13RealizationJCTA}. In the polytope model, we provide an inductive formula to obtain an arbitrary element from the highest weight element, which is non-canonical in the sense that it involves the raising crystal operators. To describe the formula, for $A\in B^{i,m}$, let $r^A$ be the row index of the first non-zero row in $A$ and $t^A$ be the least column index of the first non-zero entry in row $r^A$ in $A.$ Now let $A\in B^{i,m}$ be arbitrary. Define $A^{(k)}=(a^{(k)}_{s,t})\in B^{i,m}$ for $r^A\le k \le n$ by setting all the rows of $A$ with row indices $<k$ to zero (see section \ref{secpoly} for more details on the convention of the row-column indexing and interpreting the formula below). Note that $A^{(r^A)}=A.$ Define for $r^A\le k\le n$, the operator $\mathcal{K}^{A^{(k)}}$ by $$\mathcal{K}^{A^{(k)}}:=\left(\prod_{j = i + 1}^{k} f_{k+i+1-j}^{\varphi_{k+i+1-j}}\right) \left(\prod_{j=t^{A^{(k)}}+1}^{i}e_{i+t^{A^{(k)}}+1-j}^{h(k)}\right) \left(f_{t}^{g(k)}\prod_{j = t^{A^{(k)}}+1}^if_j^{\varphi_j}\right),$$
where
$$h(k) = \epsilon_{\substack{i + t^{A^{(k)}} + 1 - j}} - a^{(k)}_{\substack{i + t^{A^{(k)}} + 1 - j},\,k}, \ \ g(k) = a^{(k)}_{\substack{t^{A^{(k)}}},\,k}.$$
\begin{theom}
    With all the notations as above, the operator $$\mathcal{K}^{A^{(r^A)}}\mathcal{K}^{A^{(r^A+1)}}\cdots \mathcal{K}^{A^{(n)}}$$ maps the highest weight element to an arbitrary $A$ in $B^{i,m}.$\qed
\end{theom}

In the last section, we provide an explicit formula to reach to any element from the highest weight element by only using the lowering crystal operators when $i\le 2$ and $i\geq n-1.$ In those cases, we also provide the explicit image of the affine crystal isomorphism between the polytope and the tableau realizations.

\medskip

\textit{The paper is organized as follows:} In section \ref{prem}, we recall the necessary definitions of crystals, tableau and polytope realization of Kirillov-Reshetikhin crystals and introduce some statistics associated to an integral point of it. In section \ref{indfor}, we prove the inductive formula to obtain any element in the crystal graph of Kirillov–-Reshetikhin crystals from the highest weight element in the polytope realization. In section \ref{lowcase}, we do the same when $i\le 2$ and $i\geq n-1$ and provide the explicit image of the affine crystal isomorphism between the polytope and the tableau realizations.\medskip

This research was greatly facilitated by experiments carried out in the computational algebra systems SageMath \cite{sagemath}.

\section{Preliminaries}\label{prem}
\subsection{} We adopt the conventions that an empty product equals $1$ and an empty sum equals $0$. We denote by $\mathbb{C}, \mathbb{R}, \mathbb{Z}, \mathbb{N}$, and $\mathbb{Z}_{+}$ the sets of complex numbers, real numbers, integers, positive integers and non-negative integers, respectively. Given $m\in \bn,$ we denote the set $\{1,2,\dots,m\}$ by $[m].$ Moreover, if $f$ denotes a condition, we define the indicator function $\delta_{f}$ by $\delta_{f}=1$ if $f$ holds and $\delta_{f}=0$ otherwise. For instance, $\delta_{\,j \geq 2}=1$ if and only if $j \geq 2$, and $\delta_{\,j \geq 2}=0$ otherwise.
\subsection{Lie algebras of type \texorpdfstring{$A_n$}{An}} Throughout this article once and forever we fix $n\in\bn.$ Let $I=\{1,2,\cdots, n\}$ be a finite index set and $A=(a_{k,\ell})_{i,j\in I}$ be the Cartan matrix of type $A_n$. Let $\lie g$ be the Lie algebra associated to the matrix $A.$ Then $\mathfrak{g}$ is the Lie algebra of traceless $(n+1)\times (n+1)$ complex matrices, that is, $\lie g=\mathfrak{sl}_{n+1}(\mathbb{C})$. Denote by $\mathfrak{h}$ the Cartan subalgebra of $\mathfrak{g}$ consisting of diagonal matrices. Let $\Pi = \{\alpha_1, \alpha_2, \dots, \alpha_n\}$ and $\Pi^\vee = \{\alpha_1^\vee, \alpha_2^\vee, \dots, \alpha_n^\vee\}$ denote the sets of simple roots and simple coroots of $\mathfrak{g}$, respectively. Let $\theta$ be the highest root and let $\omega_1, \omega_2, \dots, \omega_n$ be the fundamental weights of $\mathfrak{g}$ . We denote by $P$ the weight lattice of $\mathfrak{g}$.\medskip

Now set $\widehat{I} := I \cup \{0\}$ and let $\widehat{A}$ be the corresponding untwisted affine Cartan matrix of type $A_n^{(1)}$. The Lie algebra $\widehat{\mathfrak{g}}$ associated with $\widehat{A}$ is the untwisted affine Lie algebra $\widehat{\lie{sl}_{n+1}}(\bc)$. Its sets of simple roots and simple coroots are $\widehat{\Pi}=\{\alpha_0,\alpha_1,\dots,\alpha_n\}$ and $\widehat{\Pi^\vee}:=\{\alpha_0^\vee,\alpha_1^\vee,\dots,\alpha_n^\vee\}$ respectively, where $\alpha_0=-\theta+\delta,$ the root $\delta$ being the null root.  We denote by $\widehat{P}$ the weight lattice of $\widehat{\lie g}.$
\subsection{Abstract crystals} A $\widehat{\lie g}$ \textit{crystal} is a non-empty set $\mathcal{B}$ together with maps $$\mathrm{wt}:\mathcal{B}\to \widehat{P},\ \ e_\ell,f_\ell:\mathcal{B}\to \mathcal{B}\cup\{0\},\ \ \epsilon_\ell,\varphi_\ell:\mathcal{B}\to \bz\cup \{-\infty\},\ \ \ell\in\widehat{I}$$ satisfying the following conditions:
    \begin{enumerate}
        \item If $x,y\in \mathcal{B}$, then $e_\ell(x)=y$ if and only if $f_\ell(y)=x$ and in this case we have $$\mathrm{wt}(y)=\mathrm{wt}(x)+\alpha_\ell,\ \ \epsilon_\ell(y)=\epsilon_\ell(x)-1,\ \ \varphi_\ell(y)=\varphi_\ell(x)+1.$$
        \item $\varphi_\ell(x)-\epsilon_\ell(x)=\mathrm{wt}(x)(\alpha_\ell^\vee)$ for $x\in\mathcal{B}$ where it is understood that if $\varphi_\ell(x)=-\infty$ then $\epsilon_\ell(x)=-\infty$ and in that case $e_\ell(x)=f_\ell(x)=0.$
    \end{enumerate}
The operators $f_\ell$ (resp. $e_\ell$), $\ell\in \widehat{I}$, are called the \textit{lowering} (resp. \textit{raising}) operators.  For brevity, we shall sometimes call only $\mathcal{B}$ as a crystal when the weight map and crystal operators are clear from the context. Throughout this article, we shall consider \textit{seminormal crystals} only, i.e. crystals where the functions $\epsilon_\ell,\varphi_\ell$ are determined by:
$$\epsilon_\ell(x):=\max\{k\in\bz_+: e_\ell^k(x)\neq 0\},\ \ \varphi_\ell(x):=\max\{k\in\bz_+: f_\ell^k(x)\neq 0\},$$ with the convention that $\epsilon_\ell(x)=-\infty$ (resp. $\varphi_\ell(x)=-\infty$) if $e_\ell(x)=0$ (resp. $f_\ell(x)=0$) for $x\in\mathcal{B}$. The crystal graph of $\mathcal{B}$ is a directed simple graph with vertex set $\mathcal{B}$ and an edge $x\stackrel{\ell}{\to}y$ whenever $f_\ell(x)=y.$  The dual crystal $\mathcal{B}^\vee$ of $\mathcal{B}$ is defined by a bijection $x \mapsto x^\vee$ between $\mathcal{B}$ and $\mathcal{B}^\vee$ such that
$$\mathrm{wt}(x^\vee)=-\mathrm{wt}(x),\ \ e_\ell(x^\vee)=f_\ell(x)^\vee,\ \ f_\ell(x^\vee)=e_\ell(x).$$
Note that the crystal graph of $\mathcal{B}^\vee$ is obtained by reversing the direction of all the arrows in the crystal graph of $\mathcal{B}.$ 

\subsection{KR crystals as tableaux}\label{seccrytab} The tableaux model for the Kirillov-Reshetikhin crystals (KR crystals in short) is standard and can be found in \cite{Kashiwara94CrystalGraph}. Let $(i,m)\in I\times \bn$. The KR crystal associated to the pair $(i,m)$ is denoted by $\mathrm{SSYT}(m\omega_i)$ and is described as follows: 
\begin{itemize}
    \item The underlying set $\mathcal{B}$ is $\mathrm{SSYT}(m\omega_i)$, consisting of all semistandard Young tableaux (SSYT in short) of rectangular shape $(m^i)$ filled with entries from $[n+1]$.
    \item For $j\in I$, the crystal operators $e_j,f_j$ are defined via the signature rule:

    \noindent Given $T\in\mathcal{B}$, let $w$ be the word obtained by reading $T$ from left to right along rows, starting at the bottom row and moving upward. Thus $T$ may be regarded as a word over the alphabet $[n+1]$. Now fix $j\in I$. In $w$, successively pair each $j+1$ with the nearest unpaired $j$ to its right, provided no letters between them are equal to $j$ or $j+1$. After all possible pairings, any unpaired $j$'s lie to the left of all unpaired $j+1$'s. Then:
\begin{itemize}
    \item If no unpaired $j$ (resp. $j+1$) remains, set $f_j(w)=0$ (resp. $e_j(w)=0$).  
    \item Otherwise, $f_j$ changes the rightmost unpaired $j$ to $j+1$, and $e_j$ changes the leftmost unpaired $j+1$ to $j$.  
\end{itemize}  
 For example, 
$$f_1(21122152731221)=21222152731221,\ \ e_1(21122152731221)=21112152731221.$$
    \item The action of the crystal operators for the affine nodes are given by (\cite{Shimozono02AffineTypeA}):
\begin{equation}\label{epphi0tab}
    e_0 = \mathrm{pr}^{-1}\circ e_1\circ \mathrm{pr},\ \ \ \ f_0 = \mathrm{pr}^{-1} \circ f_1 \circ \mathrm{pr}.
\end{equation}
where $\mathrm{pr}$ is the Sch\"{u}tzenberger’s promotion operator. The image of the promotion (resp. inverse promotion) operator applied to a tableau $T$ is computed as follows:
\begin{enumerate}
    \item Remove all entries equal to $n+1$ (resp. $1$) from $T$.
    \item Apply jeu-de-taquin to slides the empty boxes to a corner.
    \item Fill the empty boxes with $0$ (resp. $n+2$).
    \item Increase (resp. decrease) all remaining entries by $1$.
\end{enumerate}
For more details about the jeu-de-taquin slides, the reader may refer to \cite[Section A1.2]{StanleyECVol2}.
\end{itemize}
We denote by $T^{(0)}\in \mathrm{SSYT}(m\omega_i)$ (when $i$ and $m$ are clear from the context) the highest weight tableau i.e. the unique element of $\mathrm{SSYT}(m\omega_i)$ satisfying $e_i(T^{(0)})=0$ for all $i\in[n].$ Then all the entries of $T^{(0)}$ in the $j$-th row are $j$ for $j\in [i].$ 

\subsection{KR crystals as polytopes}\label{secpoly}
The polytope models for the KR crystals were introduced in \cite{Kus13RealizationJCTA} and \cite{Kus16KirillovJACO}. This model was motivated by the polytope realization of the FFL basis defined in \cite{Feigin11PBW}. Let $i\in I.$ Let $\widetilde{B}^{i}_n$ be the set of all patterns of size  with non-negative integer coefficients with the following indexing of rows and columns:
\begin{equation}\label{defA}
    A=\begin{array}{|c|c|c|c|c|}\hline
   a_{1,i}  & a_{2,i} & \cdots& a_{i-1,i}  & a_{i,i} \\ \hline
    a_{1,i+1}  & a_{2,i+1} & \cdots &a_{i-1,i+1}  & a_{i,i+1} \\ \hline
    \vdots & \vdots & \vdots & \vdots &\vdots \\\hline
    a_{1,n}  & a_{2,n} & \cdots& a_{i-1,n}  & a_{i,n} \\\hline
\end{array}
\end{equation}
A \textit{Dyck path} in $\widetilde{B}^{i}_n$ is a sequence $(\beta(1),\beta(2),\dots,\beta(n))$ such that $\beta(1)=(1,i),\beta(n)=(i,n)$ and if $\beta(r)=(a,b),$ then $\beta(r+1)\in \{(a,b+1),(a+1,b)\}.$ A subsequence of a Dyck path is called a \textit{sub Dyck path}. For $m\in \bn,$ define $B^{i,m}_n$ to be the subset of $\tilde{B}^{i}_n$ given by 
\begin{equation}\label{defpoly}
    B^{i,m}_n:=\left\{A\in \widetilde{B}^{i,m}_n: \sum_{r=1}^n a_{\beta(r)}\le m\ \ \text{ for all Dyck paths }\beta\text{ in }\widetilde{B}^{i,m}_n\right\}.
\end{equation}
We shall often suppress $n$ from $B^{i,m}_n$ and write $B^{i,m}$ instead when $n$ is clear from the context. The set $B^{i,m}$ is indeed a polytope in $\br_+^{i(n-i+1)}$ (see \cite{Kus13RealizationJCTA}).
Let $A\in B^{i,m}$ be given by \eqref{defA}. Define the weight function $\mathrm{wt}:B^{i,m}\to P$ as 
\begin{equation*}\label{weight}
\mathrm{wt}(A)=m\omega_i-\sum_{1\leq p\leq i, i\leq q\leq n}a_{p,q}\alpha_{p,q}.
\end{equation*}
It is often convenient to set $a_{k,\ell}=0$ if $k\notin[i]$ or $\ell\notin \{i,i+1,\dots,n\}.$ We now define the crystal operators following \cite{Kus13RealizationJCTA}. For $\ell\in I$, set:
\begin{equation*}\label{1}
p^\ell_+(A)=\min \left\{1\leq p\leq i|\sum^p_{j=1}a_{j,\ell-1}+\sum^{i}_{j=p}a_{j,\ell}=\max_{1\leq q\leq i}\{\sum^q_{j=1}a_{j,\ell-1}+\sum^{i}_{j=q}a_{j,\ell}\}\right\}
\end{equation*}
\begin{equation*}\label{2}
q^\ell_+(A)=\max\{1\leq p\leq i|\sum^p_{j=1}a_{j,\ell-1}+\sum^{i}_{j=p}a_{j,\ell}=\max_{1\leq q\leq i}\{\sum^q_{j=1}a_{j,\ell-1}+\sum^{i}_{j=q}a_{j,\ell}\}\}
\end{equation*}
\begin{equation*}\label{3}
p^\ell_-(A)=\max\{i\leq p\leq n|\sum^p_{j=i}a_{\ell,j}+\sum^{n}_{j=p}a_{\ell+1,j}=\max_{i\leq q\leq n}\{\sum^q_{j=i}a_{\ell,j}+\sum^{n}_{j=q}a_{\ell+1,j}\}\}
\end{equation*}
\begin{equation*}\label{4}
q^\ell_-(A)=\min\{i\leq p\leq n|\sum^p_{j=i}a_{\ell,j}+\sum^{n}_{j=p}a_{\ell+1,j}=\max_{i\leq q\leq n}\{\sum^q_{j=i}a_{\ell,j}+\sum^{n}_{j=q}a_{\ell+1,j}\}\}.
\end{equation*}

For $\ell\in I,$ the functions $\varphi_\ell,\epsilon_\ell: B^{i,m}\to \bz_{\geq 0}$ are defined by 

\begin{equation*}\label{phfi}
\varphi_l(A)=\begin{cases}
  m-\sum^{i-1}_{j=1} a_{j,i}-\sum^{n}_{j=i} a_{i,j},  & \text{if $l=i$}\\
  \sum^{p^l_+(A)}_{j=1}a_{j,l-1}-\sum^{p^l_+(A)-1}_{j=1}a_{j,l}, & \text{if $l>i$}\\
  \sum^{n}_{j=p^l_{-}(A)}a_{l+1,j}-\sum^{n}_{j=p^l_{-}(A)+1}a_{l,j}, & \text{if $l<i$}
\end{cases}
\end{equation*}
and
\begin{equation*}\label{epfin}
\epsilon_l(A)=\begin{cases}
  a_{i,i},  & \text{if $l=i$}\\
  \sum^{i}_{j=q^l_+(A)}a_{j,l}-\sum^{i}_{j=q^l_{+}(A)+1}a_{j,l-1}, & \text{if $l>i$}\\
  \sum^{q^l_{-}(A)}_{j=i}a_{l,j}-\sum^{q^l_{-}(A)-1}_{j=i}a_{l+1,j}, & \text{if $l<i$,}
\end{cases}
\end{equation*}

Now we can define the crystal operators. The images ${f}_{\ell}(A)$ and  ${e}_{\ell}(A)$ are defined to be $0$ if $\varphi_{l}(A)=0$ and $\epsilon_{l}(A)=0$ respectively. Otherwise, the image of $A$ under ${f}_{\ell}(A)$ and  ${e}_{\ell}(A)$ can be obtained from $A$ by the following rule:
\begin{equation*}\label{kashf}
{f}_{l}A=\left\{
\begin{array}{ll}
\text{ replace } a_{i,i} \text{ by } a_{i,i}+1, & \text{if } l=i,\\
\text{replace }a_{p^l_+(A),l-1} \text{ by } 
a_{p^l_+(A),l-1}-1 \text{ and } 
a_{p^l_+(A),l} \text{ by }
a_{p^l_+(A),l}+1, & \text{if } l>i,\\
\text{replace }a_{l,p^l_-(A)}\text{ by } a_{l,p^l_-(A)}+1 \text{ and } 
 a_{l+1,p^l_-(A)}\text{ by } a_{l+1,p^l_-(A)}-1,& \text{ if } l<i.\\
\end{array}\right.
\end{equation*}

\begin{equation*}\label{kashe}
{e}_{l}A=\left\{
\begin{array}{ll}
\text{replace }
a_{i,i}\text{ by } a_{i,i}-1,& \text{if } l=i,\\
\text{replace } a_{q^l_+(A),l-1} \text{ by } a_{q^l_+(A),l-1}+1 \text{ and }
a_{q^l_+(A),l} \text{ by } a_{q^l_+(A),l}-1,& \text{ if } l>i,\\
\text{replace } a_{l,q^l_-(A)} \text{ by } a_{l,q^l_-(A)}-1 \text{ and }
a_{l+1,q^l_-(A)} \text{ by } a_{l+1,q^l_-(A)}+1,& \text{ if } l<i.\\
\end{array}\right.
\end{equation*}
We describe the crystal operators and the functions $\epsilon_\ell$, $\varphi_\ell$ using the diagrams in Figures~\ref{figflei} and~\ref{figflgri}, as well as Figures~\ref{figelei} and~\ref{figelgri}. These visual descriptions will be used throughout the article. We shall only refer to the relevant figure to describe the crystal operators and the functions $\varphi_\ell$ and $\epsilon_\ell.$ We shall only explain Figure~\ref{figflei}, which describes the action of the operator $f_\ell$ for $\ell < i$ on a $A\in B^{i,m}$. The remaining cases are analogous.

We consider columns $\ell$ and $\ell+1$ of $A$ and look at the paths of the form described in the figure (formally defined as step paths in the next subsection) restricted to these two columns, where the sum of the entries along the path is maximal. Let $p_{-}^\ell(A)$ denote the largest row index at which this maximal sum occurs; the fact that we choose the maximum row index is indicated in the figure by a downward arrow ($\downarrow$). The value $\varphi_\ell(A)$ is computed as the difference between two partial column sums: the sum of the entries in column $\ell+1$ starting at row $p_{-}^\ell(A)$, minus the sum of the entries in column $\ell$ starting at row $p_{-}^\ell(A)+1$. In the diagram, these correspond to the blue segment ({\color{blue}$|$}) and the red segment ({\color{red}$|$}), respectively.
Finally, if $\varphi_\ell(A) > 0$, then the operator $f_\ell$ acts by increasing the entry at position $(\ell,p_{-}^\ell(A))$ by $1$ and decreasing the entry at position $(\ell+1,p_{-}^\ell(A))$ by $1$.
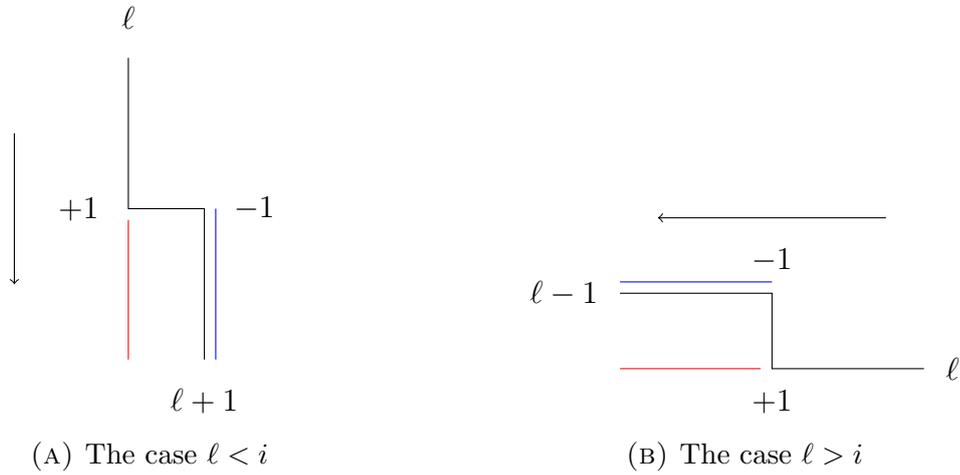
\begin{figure}[ht]
 \begin{subfigure}{0.49\textwidth}
    \centering
    \begin{tikzpicture}\label{fle}
    \draw[] (0,2) -- (0,0) -- (1,0) -- (1,-2);
    \node[]  at (0,2.1) [label=above:$\ell$]{};
    \node[]  at (1,-2.1) [label=below:$\ell+1$]{};
    \node[] at (-0.1,0) [label=left:$+1$]{};
    \node[] at (1.1,0) [label=right:$-1$]{};
    \draw[->] (-1.5,1)--(-1.5,-1);
    \draw[red] (0,-0.15)--(0,-2);
    \draw[blue] (1.15,0)--(1.15,-2);
\end{tikzpicture}
    \caption{The case $\ell<i$}
    \label{figflei}
    \end{subfigure}
    \begin{subfigure}{0.49\textwidth}
\centering
\begin{tikzpicture}\label{fge}
    \draw[] (-2,0) -- (0,0) -- (0,-1) -- (2,-1);
    \node[]  at (-2,0) [label=left:$\ell-1$]{};
    \node[]  at (2,-1) [label=right:$\ell$]{};
    \node[] at (0,0) [label=above:$-1$]{};
    \node[] at (0,-1) [label=below:$+1$]{};
    \draw[blue] (-2,0.15)--(0,0.15);
    \draw[red] (-2,-1)--(-0.15,-1);
    \draw[->] (1.5,1)--(-1.5,1);
\end{tikzpicture}
    \caption{The case $\ell>i$}
    \label{figflgri}
    \end{subfigure}
    \caption{Action of $f_\ell$ operators}
\end{figure}

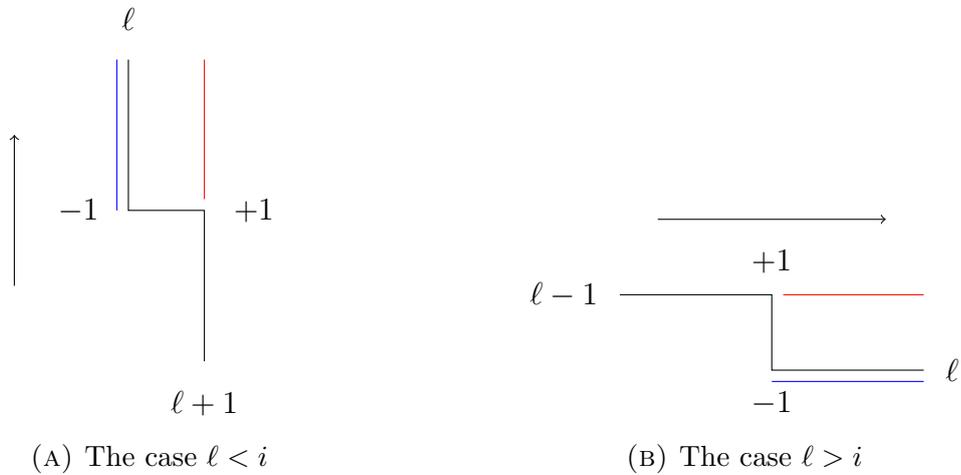
\begin{figure}[ht]
 \begin{subfigure}{0.49\textwidth}

    \centering
    \begin{tikzpicture}\label{ele}
    \draw[] (0,2) -- (0,0) -- (1,0) -- (1,-2);
    \node[]  at (0,2.1) [label=above:$\ell$]{};
    \node[]  at (1,-2.1) [label=below:$\ell+1$]{};
    \node[] at (-0.1,0) [label=left:$-1$]{};
    \node[] at (1.1,0) [label=right:$+1$]{};
    \draw[->] (-1.5,-1)--(-1.5,1);
    \draw[blue] (-0.15,0)--(-0.15,2);
    \draw[red] (1,0.15)--(1,2);
\end{tikzpicture}
    \caption{The case $\ell<i$}
    \label{figelei}
        \end{subfigure}
    \begin{subfigure}{0.49\textwidth}

\centering
\begin{tikzpicture}\label{ege}
    \draw[] (-2,0) -- (0,0) -- (0,-1) -- (2,-1);
    \node[]  at (-2,0) [label=left:$\ell-1$]{};
    \node[]  at (2,-1) [label=right:$\ell$]{};
    \node[] at (0,0) [label=above:$+1$]{};
    \node[] at (0,-1) [label=below:$-1$]{};
    \draw[blue] (0,-1.15)--(2,-1.15);
    \draw[red] (0.15,0)--(2,0);
    \draw[->] (-1.5,1)--(1.5,1);
\end{tikzpicture}
    \caption{The case $\ell>i$}
    \label{figelgri}
    \end{subfigure}
    \label{figel}
    \caption{Action of $e_\ell$ operators}
\end{figure}
We denote by $A^{(0)}\in B^{i,m}$ (when $i$ and $m$ are clear from the context) the highest weight element i.e. the unique element of $B^{i,m}$ satisfying $e_i(A^{(0)})=0$ for all $i=1,2,\dots,n.$ Then all the entries of $A^{(0)}$ are zero. Now define the crystal operators $e_0$ and $f_0$ by (\cite{Kus16KirillovJACO}):
\begin{equation}\label{epphi0}
    \varphi_0(A)=a_{1,n},\ \ \epsilon_0(A)= m - \sum_{j=i}^n a_{1,j}-\sum_{j=2}^i a_{j,n}.
\end{equation}
Set $f_0(A)=0$ (resp. $e_0(A)=0$) if $\varphi_0(A)=0$ (resp. $\epsilon_0(A)=0$). Otherwise $f_0(A)$ (resp. $e_0(A)$) is obtained from $A$ by replacing $a_{1,n}$ by $a_{1,n}-1$ (resp. $a_{1,n}+1$). The following theorem was proved in \cite{Kus13RealizationJCTA,Kus16KirillovJACO}.
\begin{thm}
    The set $B^{i,m}$ with the crystal operators defined above is isomorphic to the Kirillov Reshetikhin crystal.
\end{thm}
For the rest of this article, we denote by
\begin{equation}\label{cryisomap}
    \varphi_{i,m}:B^{i,m}\to\mathrm{SSYT(m\omega_i)}
\end{equation}
the affine crystal isomorphism between the polytope and the tableau realizations.\medskip

\noindent {\em Throughout this article, we adopt the following convention:} Let $A$ be either an element of $B^{i,m}$ or a tableau and let $a$ be a non-negative integer. We set $f_j^{\varphi_j - a}(A)$ to be $f_j^{\varphi_j(A) - a}(A),$ when $a\le \varphi_j(A)$ and undefined otherwise. A product of the form $\prod_{j=s}^t f_j^{\varphi_j - a_j}(A)$ is interpreted as the composition $$f_s^{\varphi_s - a_s} f_{s+1}^{\varphi_{s+1} - a_{s+1}} \cdots f_t^{\varphi_t - a_t}(A).$$
The product is understood to be empty whenever $s > t$. A similar convention is adopted for the raising operators as well.
\subsection{Statistics on integral points of the polytope}\label{statpoly} In this subsection, we introduce some statistics associated with the integral points of the FFL polytopes. Let $A=(a_{k,\ell})\in B^{i,m}$ and fix $1\le s< i\le r\le n.$ We define the \textit{step path} $p_{s,r}$ as the sequence of positions:
$$p_{s,r}=\{(s,i),(s,i+1), \dots ,(s,r),(s+1,r),(s+1,r+1),\dots,(s+1,n)\}.$$ We then set
$$S^A(s,r):=\sum_{x\in p_{s,r}} a_x\ \  \text{ and }M^A(s):=\max\{S^A(s,r): i\le r\le n\}.$$ 
Now we define a sequence $R^A(s) = \{r^A(s,0), r^A(s,1), \dots,r^A(s,\ell^A(s))\}$ as follows: Set $r^A(s,0) = i-1$ and define inductively
\begin{equation}\label{defr}
    r^A(s,k+1):=\max\{r>r^A(s,k): S^A(s,r)=\max\{S^A(s,p): r^A(s,k)<p\le n\}\}
\end{equation}
with $\ell^A(s)$ such that $r^A(s,\ell^A(s)) = n$.

Finally, we define some non-negative integers $x^A(s,r)$. Set $x^A(s,r)=0$ if $r\notin R^A(s)\backslash\{i-1\}$ and for $1\le k<\ell^A(s)$ the integers $x^A(s,r)$ satisfy:
\begin{equation}\label{recx}
    S^A(s,r^A(s,k+1))+\sum_{p=i}^{r^A(s,k)} x^A(s,p)=M^A(s),
\end{equation}
and we set $x^A(s,n)=a_{s+1,n}.$ 

For example, let $m$ be large and take $(i,n)=(3,6).$ Let $A\in B^{i,m}$ be given by
$$A=\begin{array}{|c|c|c|}\hline 2 & 0 & 2\\ \hline 0 & 1 & 2 \\ \hline 1 & 2 & 1\\ \hline 1& 0 &5\\ \hline \end{array}$$
Then we have $$S^A(1,4)=5,\ \ S^A(2,5)=9,\ \ S^A(2,3)=10.$$ Moreover,
$$R^A(1)=\{2,5,6\},\ \ \ R^A(2)=\{2,3,5,6\}$$ and $$x^A(2,3)=x^A(2,5)=x^A(2,6)=x^A(1,5)=1,$$ $$x^A(2,4)=x^A(1,3)=x^A(1,4)=x^A(1,6)=0.$$
\begin{rem}\label{xandr}
    By definition, we have $x^A(s,p)\in \bz_+$ for $i\le p\le n.$ Moreover, we have $x^A(s,j)>0$ if and only if $j\in \{r^A(s,p):1\le p< \ell^A(s)+\delta_{(a_{s+1,n}>0)}\}$. In particular, $r^A(s,1)=i$ if  $x^A(s,i)>0.$
\end{rem}
  We begin with an easy lemma which will be needed later. 
\begin{lem}\label{actionrow}
    Let $1\le k<\ell^A(s).$ For any $0\le x< x^A(s,r^A(s,k))$ and $r^A(s,k)<r\le n$, we have $$S^A(s,r)+\sum_{p=i}^{r^A(s,k-1)} x^A(s,p)+x<M^A(s).$$ 
\end{lem}
\begin{proof}
Since by definition we have $S^A(s,r)\le S^A(s,r^A(s,k+1))$ for all $r^A(s,k)<r\le n$, the result follows from Eq.\eqref{recx}.
\end{proof}
The following proposition will be used throughout this section.
\begin{prop}\label{sumx}
    Let $A\in B^{i,m}$ and let $1\le s< i,\ 1\le r\le \ell^A(s).$ Then we have 
    \begin{equation*}
        \sum_{p=r^A(s,r)}^n x^A(s,p)=S^A(s,r^A(s,r))-S^A(s,n)+a_{s+1,n}.
    \end{equation*}
\end{prop}
\begin{proof}
    Eq.\eqref{recx} can be encoded in a matrix equation $RX=B$ where 
$R=(r_{k,\ell})$ is a square matrix of order $\ell^A(s)-1$, $X=(x_1,\dots, x_{\ell^A(s)-1})$ and $B=(b_1,b_2,\dots,b_{\ell^A(s)-1})$ are column vectors so that
$$r_{k,\ell}=\begin{cases}
    1 & \text{ if } k\geq \ell,\\
    0 & \text{ elsewhere }
\end{cases},\ \ \ x_k=x^A(s,r^A(s,k)),\ \ \ b_k=M^A(s)-S^A(s,r^A(s,k+1)).$$

Solving the system by inverting $R$ we obtain:
\begin{equation}\label{valx}
    x^A(s,r^A(s,r))=S^A(s,r^A(s,r))-S^A(s,r^A(s,r+1)),\ \ 1\le r\le \ell^A(s)-1.
\end{equation}
Now the result follows along with the fact that $x^A(s,\ell^A(s))=a_{s+1,n}.$
\end{proof}
\begin{cor}\label{sumphi}
    We have $\sum_{p=1}^{n}x^A(s,p)=\varphi_s(A).$
\end{cor}
\begin{proof} 
Proof follows from proposition \ref{sumx} and the definition of $\varphi_i$ in section \ref{secpoly} since $$S^A(s,r^A(s,1))-S^A(s,n)+a_{s+1,n}=\varphi_{s}(A).$$
\end{proof}
\begin{rem}\label{remfphi}
    Let $s<i.$ Using lemma \ref{actionrow} and corollary \ref{sumphi}, we have that $B:=f_s^{\varphi_s}(A)=(b_{k,\ell})$ is given by
    $$b_{k,\ell}=a_{k,\ell}+\delta_{(k=s)}x^A(s,\ell)-\delta_{(k=s+1)}x^A(s,\ell).$$
\end{rem}
\section{An inductive formula for the polytope model}\label{indfor}
In this section we shall prove an inductive formula to obtain an arbitrary element from the highest weight element in the polytope realization. Throughout the section, we fix the following notation: rows (resp. columns) of an element in $B^{i,m}$ are indexed by $\{i,i+1,\dots,n\}$ (resp. $\{1,2,\dots,i\}$).  Let $Q = (q_{s,t}) \in B^{i,m}$ be an arbitrary element and suppose that the $k$-th row is the first non-zero row of $Q$. Let $t^{Q}$ be the minimum column index such that $q_{t^{Q},k}\neq 0.$ Let $R=(r_{s,t})$ be obtained from $Q$ by replacing all entries in the $k$-th row of $Q$ with zero. Therefore, $q_{s,t}=r_{s,t}$ for all $t\geq k+1.$ The main result of this section is the following theorem.
\begin{thm}\label{thmpolytope}
    The function $\mathcal{K}^Q:=HEF$ maps $R$ to $Q$ where 
    $$H=\prod_{j = i + 1}^{k} f_{k+i+1-j}^{\varphi_{k+i+1-j}},\ \ E=\prod_{j=t^Q+1}^{i}e_{i+t^Q+1-j}^{\epsilon_{i+t^Q+1-j}-q_{i+t^Q+1-j,k}},\ \ F=f_{t}^{q_{t^Q,k}}\prod_{j = t^Q+1}^if_j^{\varphi_j}.\qed$$
\end{thm}
When $i=1$, the formula simplifies to 
$$\mathcal{K}^Q = \left(\prod_{j = 2}^{k} f_{k+2-j}^{\varphi_{k+2-j}}\right)f_1^{q_{1,k}}.$$ Using the definition of the crystal operators in section \ref{secpoly}, it is immediate that $\mathcal{K}^Q$ maps $Q$ to $R$. Before turning to the case $i>1$, we illustrate theorem~\ref{thmpolytope} with a concrete example.
\begin{example}
    Let $(m,i,n)=(9,3,5)$ and consider
    $$A=\begin{array}{|c|c|c|}\hline 0 & 1 & 1\\ \hline 1 & 3 & 4 \\ \hline 1 & 3 & 1\\\hline \end{array}$$
We apply theorem~\ref{thmpolytope} to construct $A$ from the highest weight element. In each inductive step, the key idea is to first create the $k$-th row of $Q$ in the first row of $R$ and then apply the appropriate crystal operators to move this row into the $k$-th row of $R$, thereby transforming $R$ into $Q$. Throughout this example, we shall use the descriptions of the action of the crystal operators in figures \ref{figflei},\ref{figflgri},\ref{figelei} and \ref{figelgri} without any further mention. \medskip

\noindent \textbf{First step:} We begin with the pair
$$Q=\begin{array}{|c|c|c|}\hline 0 & 0 & 0\\ \hline 0 & 0& 0 \\ \hline 1 & 3 & 1\\\hline \end{array}\ ,\ \ \ R= \begin{array}{|c|c|c|}\hline 0 & 0 & 0\\ \hline 0 & 0& 0 \\ \hline 0 & 0 & 0\\\hline \end{array}.$$
Here $t^Q=1$. The path from $R$ to $Q$ is given by:\medskip

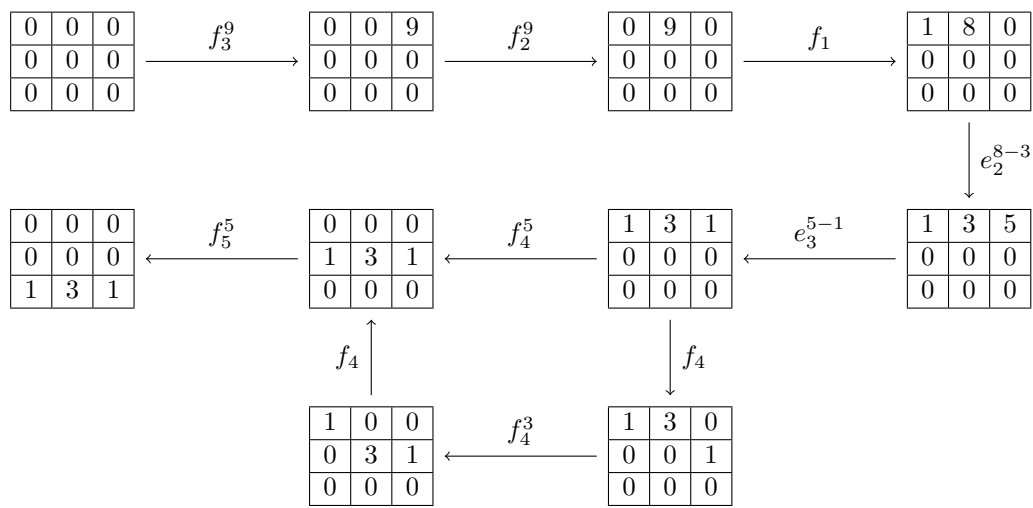
\begin{tikzpicture}[every node/.style={inner sep=4pt, font=\footnotesize}]
    \node (v1) {$\begin{array}{|c|c|c|} \hline
    0 & 0 & 0 \\ \hline
    0 & 0 & 0 \\ \hline
    0 & 0 & 0 \\ \hline
    \end{array}$};
    
    \node[right=2cm of v1] (v2) {$\begin{array}{|c|c|c|} \hline
    0 & 0 & 9 \\ \hline
    0 & 0 & 0 \\ \hline
    0 & 0 & 0 \\ \hline
    \end{array}$};
    
    \node[right=2cm of v2] (v3) {$\begin{array}{|c|c|c|} \hline
    0 & 9 & 0 \\ \hline
    0 & 0 & 0 \\ \hline
    0 & 0 & 0 \\ \hline
    \end{array}$};
    
    \node[right=2cm of v3] (v4) {$\begin{array}{|c|c|c|} \hline
    1 & 8 & 0 \\ \hline
    0 & 0 & 0 \\ \hline
    0 & 0 & 0 \\ \hline
    \end{array}$};
    
    \node[below=1cm of v4] (v5) {$\begin{array}{|c|c|c|} \hline
    1 & 3 & 5 \\ \hline
    0 & 0 & 0 \\ \hline
    0 & 0 & 0 \\ \hline
    \end{array}$};

    \node[left=2cm of v5] (v05) {$\begin{array}{|c|c|c|} \hline
    1 & 3 & 1 \\ \hline
    0 & 0 & 0 \\ \hline
    0 & 0 & 0 \\ \hline
    \end{array}$};

    \node[below=1cm of v05] (v051) {$\begin{array}{|c|c|c|} \hline
    1 & 3 & 0 \\ \hline
    0 & 0 & 1 \\ \hline
    0 & 0 & 0 \\ \hline
    \end{array}$};

    \node[left=2cm of v051] (v052) {$\begin{array}{|c|c|c|} \hline
    1 & 0 & 0 \\ \hline
    0 & 3 & 1 \\ \hline
    0 & 0 & 0 \\ \hline
    \end{array}$};

    \node[left=2cm of v05] (v6) {$\begin{array}{|c|c|c|} \hline
    0 & 0 & 0 \\ \hline
    1 & 3 & 1 \\ \hline
    0 & 0 & 0 \\ \hline
    \end{array}$};

    \node[left=2cm of v6] (v7) {$\begin{array}{|c|c|c|} \hline
    0 & 0 & 0 \\ \hline
    0 & 0 & 0 \\ \hline
    1 & 3 & 1 \\ \hline
    \end{array}$};

    \draw[->] (v1) -- node[midway, above] {$f_3^{9}$} (v2);
    \draw[->] (v2) -- node[midway, above] {$f_2^{9}$} (v3);
    \draw[->] (v3) -- node[midway, above] {$f_1$} (v4);
    \draw[->] (v4) -- node[midway, right] {$e_2^{8-3}$} (v5);
    \draw[->] (v5) -- node[midway, above] {$e_3^{5-1}$} (v05);
    \draw[->] (v05) -- node[midway, above] {$f_4^{5}$} (v6);
    \draw[->] (v6) -- node[midway, above] {$f_5^{5}$} (v7);
    \draw[->] (v05) -- node[midway, right] {$f_4$} (v051);
    \draw[->] (v051) -- node[midway, above] {$f_4^3$} (v052);
    \draw[->] (v052) -- node[midway, left] {$f_4$} (v6);
\end{tikzpicture}

\medskip

\textbf{Second step:} Next we consider
$$Q=\begin{array}{|c|c|c|}\hline 0 & 0 & 0\\ \hline 1 & 3& 4 \\ \hline 1 & 3 & 1\\\hline \end{array}\ ,\ \ \ R= \begin{array}{|c|c|c|}\hline 0 & 0 & 0\\ \hline 0 & 0& 0 \\ \hline 1 & 3 & 1\\\hline \end{array}.$$
Again we have $t^Q=1.$ The path is given by:\medskip

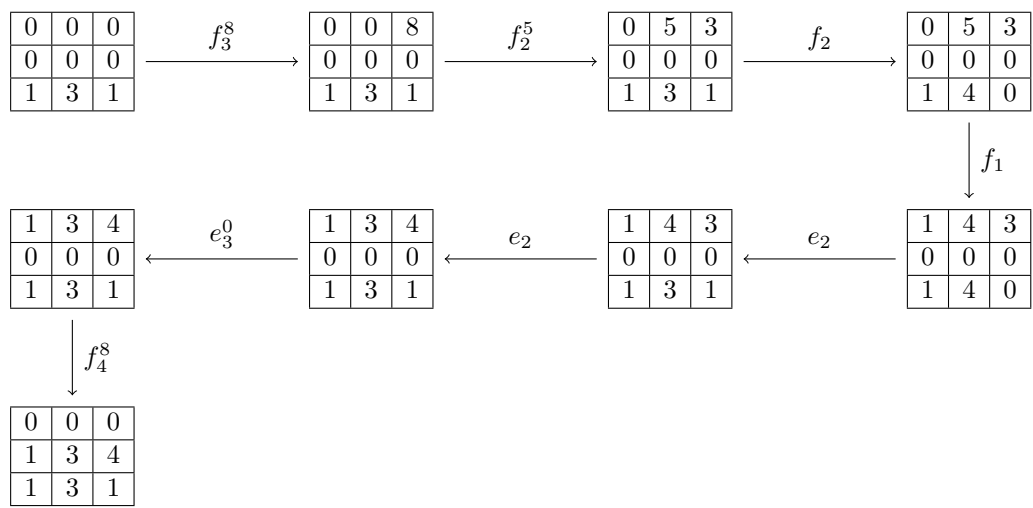
\begin{tikzpicture}[every node/.style={inner sep=4pt, font=\footnotesize}]
    \node (v1) {$\begin{array}{|c|c|c|} \hline
    0 & 0 & 0 \\ \hline
    0 & 0 & 0 \\ \hline
    1 & 3 & 1 \\ \hline
    \end{array}$};
    
    \node[right=2cm of v1] (v2) {$\begin{array}{|c|c|c|} \hline
    0 & 0 & 8 \\ \hline
    0 & 0 & 0 \\ \hline
    1 & 3 & 1 \\ \hline
    \end{array}$};
    
    \node[right=2cm of v2] (v3) {$\begin{array}{|c|c|c|} \hline
    0 & 5 & 3 \\ \hline
    0 & 0 & 0 \\ \hline
    1 & 3 & 1\\ \hline
    \end{array}$};
    
    \node[right=2cm of v3] (v4) {$\begin{array}{|c|c|c|} \hline
    0 & 5 & 3 \\ \hline
    0 & 0 & 0 \\ \hline
    1 & 4 & 0 \\ \hline
    \end{array}$};
    
    \node[below=1cm of v4] (v5) {$\begin{array}{|c|c|c|} \hline
    1 & 4 & 3\\ \hline
    0 & 0 & 0 \\ \hline
    1 & 4 & 0  \\ \hline
    \end{array}$};

    \node[left=2cm of v5] (v05) {$\begin{array}{|c|c|c|} \hline
    1 & 4 & 3\\ \hline
    0 & 0 & 0 \\ \hline
    1 & 3 & 1  \\ \hline
    \end{array}$};

    \node[left=2cm of v05] (v6) {$\begin{array}{|c|c|c|} \hline
    1 & 3 & 4\\ \hline
    0 & 0 & 0 \\ \hline
    1 & 3 & 1  \\ \hline
    \end{array}$};

    \node[left=2cm of v6] (v7) {$\begin{array}{|c|c|c|} \hline
   1 & 3 & 4\\ \hline
    0 & 0 & 0 \\ \hline
    1 & 3 & 1  \\ \hline
    \end{array}$};
    \node[below=1cm of v7] (v8) {$\begin{array}{|c|c|c|} \hline
    0 & 0 & 0 \\ \hline
    1 & 3 & 4\\ \hline
    1 & 3 & 1  \\ \hline
    \end{array}$};

    \draw[->] (v1) -- node[midway, above] {$f_3^{8}$} (v2);
    \draw[->] (v2) -- node[midway, above] {$f_2^{5}$} (v3);
    \draw[->] (v3) -- node[midway, above] {$f_2$} (v4);
    \draw[->] (v4) -- node[midway, right] {$f_1$} (v5);
    \draw[->] (v5) -- node[midway, above] {$e_2$} (v05);
    \draw[->] (v05) -- node[midway, above] {$e_2$} (v6);
    \draw[->] (v6) -- node[midway, above] {$e_3^0$} (v7);
    \draw[->] (v7) -- node[midway, right] {$f_4^8$} (v8);
\end{tikzpicture}

Let $B=f_3^8(R)$. Then $\varphi_2(B)=6$. For $k\le 4$, we have $$\max\{r: S^{f_2^k(B)}(2,r)=9\}=3,$$ while for $k=5$ the maximum is attained at $r=5$. Thus $f_2$ acts on the row indexed by $3$ five times and subsequently on the row indexed by $5$ once. Similarly, if $C=f_1f_2^6f_3^8(R),$ then $$\min\{r: S^{C}(2,r)=8\}=5,\ \ \min\{r: S^{e_2(C)}(2,r)=8\}=3.$$ This shows that although $f_1$ introduces changes in the column indexed by $2$, the subsequent action of $e_2$ precisely cancels the disturbances (the changes in the entries already constructed in the previous inductive steps) introduced by $f_2$. We prove that this cancellation holds in general for all raising and lowering operators involved.\medskip

\noindent\textbf{Third step:} Finally, consider
$$Q=\begin{array}{|c|c|c|}\hline 0 & 1 & 1\\ \hline 1 & 3& 4 \\ \hline 1 & 3 & 1\\\hline \end{array}\ ,\ \ \ R= \begin{array}{|c|c|c|}\hline 0 & 0 & 0\\ \hline 1 & 3& 4 \\ \hline 1 & 3 & 1\\\hline \end{array}.$$
with $t^Q=2.$ The corresponding path is:\medskip

\begin{tikzpicture}[every node/.style={inner sep=4pt, font=\footnotesize}]
    \node (v1) {$\begin{array}{|c|c|c|} \hline
    0 & 0 & 0\\ \hline 1 & 3& 4 \\ \hline 1 & 3 & 1 \\ \hline
    \end{array}$};
    
    \node[right=2cm of v1] (v2) {$\begin{array}{|c|c|c|} \hline
    0 & 0 & 4\\ \hline 1 & 3& 4 \\ \hline 1 & 3 & 1\\ \hline
    \end{array}$};
    
    \node[right=2cm of v2] (v3) {$\begin{array}{|c|c|c|} \hline
    0 & 1 & 3\\ \hline 1 & 3& 4 \\ \hline 1 & 3 & 1\\ \hline
    \end{array}$};
    
    \node[right=2cm of v3] (v4) {$\begin{array}{|c|c|c|} \hline
   0 & 1 & 1\\ \hline 1 & 3& 4 \\ \hline 1 & 3 & 1\\ \hline
    \end{array}$};

    \draw[->] (v1) -- node[midway, above] {$f_3^{4}$} (v2);
    \draw[->] (v2) -- node[midway, above] {$f_1$} (v3);
    \draw[->] (v3) -- node[midway, above] {$e_3^{3-1}$} (v4);
\end{tikzpicture}

\noindent Therefore, by successive application of the inductive algorithm of theorem~\ref{thmpolytope}, we obtain a path from the highest weight element to $A$. \qed
\end{example}

\noindent For the rest of the section, we assume that $i > 1$. We set $$R(i) := f_i^{\varphi_i}(R),\ \ R(p):=f_{\ell}^{\varphi_{s}}(R(p+1))\ \text{ for }2\le p \le i-1.$$
Also for notational convenience, we write $$S^p(j),\ \  x^p(j),\ \ r^p(j),\ \ \ell^p$$ instead of $$S^{R(p)}(p-1,j),\  x^{R(p)}(p-1,j),\ r^{R(p)}(p-1,j),\ \ell^{R(p)}(p-1)$$ respectively for $2\le p\le i.$ 
\subsection{Action of the operator \texorpdfstring{$F$}{F}} Applying remark \ref{remfphi} successively, for $2\le p\le i,$ the element $R(p)$ is given by
\begin{equation}\label{polrs}
    R(p)_{s,t}=\begin{cases}
    m-\sum_{j=k+1}^n r_{i,j}-x^i(i)& \text{ if } (s,t)=(i,i),\\
    r_{s,t}+\delta_{(p\le s<i)}x^{s+1}(\ell)-\delta_{(p<s\le i)}x^s(\ell) & \text{ else.}
\end{cases}
\end{equation}
    Also, since by definition $x^i(n)=r_{i,n}$, we obtain inductively
    \begin{equation}\label{valxn}
        x^p(n)=\sum_{j=p}^i r_{j,n},\ \ \ 2\le p\le i.
    \end{equation}
In particular, we have $R(p)_{s,n}=0$ for all $s\geq p+1.$ The following easy lemma is very useful.
\begin{lem}\label{maxm}
    For all $2\le p\le i$, we have $\sum_{j=i}^n R(p)_{p,j}=m.$ In particular, we have $S^{p}(i)=M^{R(p)}(p-1)=m.$
\end{lem}
\begin{proof}
    Clearly the result holds for $p=i.$ By induction we assume that the result holds for $p+1$ for $2\le p<i.$ Hence $\sum_{j=i}^nR(p+1)_{p+1,j}=m$ and therefore $S^{p+1}(i)=m.$ Now applying Eq.\eqref{recx} for $A=R(p+1)$ and $k=\ell^{p+1}-1,$ we obtain that $$S^{p+1}(n)+\sum_{j=i}^{n-1}x^{p+1}(j)=m.$$ 
    Now the result follows from Eq.\eqref{polrs} and Eq.\eqref{valxn} since 
    $$S^{p+1}(n)=\sum_{j=i}^n r_{p,j}+r_{p+1,n}+x^{p+2}(n)=\sum_{j=i}^n r_{p,j}+x^{p+1}(n).$$
\end{proof}
We now prove the following crucial proposition. 
\begin{prop}\label{propF}
    For any $t^Q+1\le s \le i$ we have
    \begin{enumerate}[label = (\alph*)]
        \item $x^{s}(i) \ge \sum_{j=1}^{s-1} q_{j,k}$,
        \item for any $2\le p_{s-1} \le \ell^s$, there exist $p_{s},p_{s+1},\dots,p_{i-1}$ such that 
        \begin{equation}\label{funnyeq1}
            r^{j+1}(p_j-1) < r^j(p_{j-1}) \le r^{j+1}(p_j),\ \ s \le j \le i-1
        \end{equation}
        and
        \begin{equation}\label{funnyeq2}
            \sum_{j=r^s(p_{s-1})}^n x^s(j) = \sum_{u=s}^{i-1}\sum_{j=r^u(p_{u-1})}^{r^{u+1}(p_u)} q_{u,j} + \sum_{j=r^i(p_{i-1})}^n q_{i,j} - \sum_{j=r^s(p_{s-1})+1}^n q_{s-1,j}.
        \end{equation}
    \end{enumerate}
\end{prop}
\begin{proof}
    To begin with, we may assume, without loss of generality, that $t^Q=1$, i.e., $q_{1,k} \neq 0$. We will prove both parts $(a)$ and $(b)$ using backward induction on $s$
    
    \noindent \textbf{Base case:} Assume that $s=i$. It is easy to see that

    $$S^i(r)=\begin{cases}
        m & \text{ if } r=i,\\
        \sum_{j=k+1}^{n}q_{i,j} & \text{ if } i < r < k+1,\\
        \sum_{j=k+1}^{r} q_{i-1,j} + \sum_{j=r}^{n}q_{i,j} & \text{ if } k+1 \le r \le n.
    \end{cases}$$
    Note that if $r\neq i,$ then $\sum_{j=1}^{i-1}q_{j,k} + S^i(r)$ is a sum of entries of a sub Dyck path in $Q$. Since $q_{1,k}>0$, we obtain
    $$S^i(r) \le m- \sum_{j=1}^{i-1}q_{j,k} < m.$$
    Hence $r^i(1) = i$ by Eq.\eqref{defr} and by Eq.\eqref{recx} we have
    $$x^i(i) = m - S^i(r^i(2))\ge m - \left(m - \sum_{j=1}^{i-1} q_{j,k}\right) = \sum_{j=1}^{i-1} q_{j,k}$$
    proving part $(a)$. Part $(b)$ follows from proposition \ref{sumx} noting the fact $r^i(s)\geq k+1$ for all $s>1$ since $r^i(1) = i$.\medskip
    
    \noindent \textbf{Induction step:} Now assume the statement holds for $s+1$ with some $2 \le s<i$. In particular, for any $p_s\geq 2$, we have 
    \begin{equation}\label{eq0}
        \sum_{j=r^{s+1}(p_{s})}^n x^{s+1}(j)  = \sum_{u=s+1}^{i-1}\sum_{j=r^u(p_{u-1})}^{r^{u+1}(p_u)} q_{u,j} + \sum_{j=r^i(p_{i-1})}^n q_{i,j}- \sum_{j=r^{s+1}(p_s)+1}^n q_{s,j}.
    \end{equation}
    for some $p_{s+1},p_{s+2},\dots,p_{i-1}$ satisfying Eq.\eqref{funnyeq1}.
    
      Note that we have $S^s(i)=m$ by lemma \ref{maxm}. Also from part (a) of induction hypothesis and remark \ref{xandr} together imply $r^{s+1}(1)=i$ and thus $r^{s+1}(2)\geq k+1$. We claim that for $i<r\le n$, we must have $$S^s(r)\le m-\sum_{j=1}^{s-1}q_{j,k}<m.$$
     \noindent\textbf{Case I:} If $i< r < k+1$, then by Eq.\eqref{polrs}
     $$S^s(r)=\sum_{j=k+1}^n q_{s,j} + \sum_{j=r^{s+1}(2)}^n x^{s+1}(j).$$
    Let $p_{s} = 2$. Then applying Eq.\eqref{eq0} by induction hypothesis, we obtain
    $$S^s(r)=\sum_{j=k+1}^{r^{s+1}(p_{s})} q_{s,j} + \sum_{u=s+1}^{i-1}\sum_{j=r^u(p_{u-1})}^{r^{u+1}(p_u)} q_{u,j} + \sum_{j=r^i(p_{i-1})}^n q_{i,j}$$
    \noindent \textbf{Case II:} If $k+1 \le r \le n$, again by using Eq.\eqref{polrs} we obtain
    $$S^{s}(r)=\sum_{j=k+1}^{r} q_{s-1,j} + \sum_{j=r}^n q_{s,j} + \sum_{j=r^{s+1}(p_{s})}^n x^{s+1}(j),$$
    where $r$ satisfies $r^{s+1}(p_{s}-1)<r\le r^{s+1}(p_{s})$. Since $r^{s+1}(p_{s})\geq k+1$ and $r^{s+1}(1)=i,$ we must have $p_s\geq 2$. Therefore, again using Eq.\eqref{eq0} by induction hypothesis, we have
    $$S^s(r)=\sum_{j=k+1}^{r} q_{s-1,j} + \sum_{j=r}^{r^{s+1}(p_{s})} q_{s,j} + \sum_{u=s+1}^{i-1}\sum_{j=r^u(p_{u-1})}^{r^{u+1}(p_u)} q_{u,j} + \sum_{j=r^i(p_{i-1})}^n q_{i,j}$$
    Therefore, we have that $\sum_{j=1}^{s-1}q_{j,k} + S^s(r)$ is the sum of entries of a sub Dyck path in $A$ for any $i < r \le n$, which proves the claim since
    \begin{equation*}
        S^s(r) \le m - \sum_{j=1}^{s-1}q_{j,k} < m.
    \end{equation*}
    This implies $r^s(1) = i$ by Eq.\eqref{defr} and therefore using Eq.\eqref{recx} we obtain
    $$x^s(i)=m - S^s(r^s(2))\ge m - \left(m - \sum_{j=1}^{s-1} q_{j,k}\right) = \sum_{j=1}^{s-1} q_{j,k}.$$
    proving part $(a)$. \medskip
    
    To prove $(b)$ assume that $2 \le p_{s-1}\le \ell^s$. By proposition \ref{sumx} and then using Eq.\eqref{polrs}, \eqref{valxn} we have
    \begin{align*}
        \sum_{j=r^s(p_{s-1})}^n x^s(j) &= S^s(r^s(p_{s-1})) - S^s(n) + x^s(n)\\
        & = \sum_{j=r^s(p_{s-1})}^n q_{s,j} + \sum_{j=r^s(p_{s-1})}^n x^{s+1}(j) - \sum_{j=r^s(p_{s-1})+1}^n q_{s-1,j}\\
        & = \sum_{j=r^s(p_{s-1})}^n q_{s,j} + \sum_{j=r^{s+1}(p_{s})}^n x^{s+1}(j) - \sum_{j=r^s(p_{s-1})+1}^n q_{s-1,j}
    \end{align*}
    where $r^{s+1}(p_{s}-1)<r^s(p_{s-1})\le r^{s+1}(p_{s})$ (such a row index exists since $r^s(p_{s-1})>i$). Since $p_s\geq 2$, the induction hypothesis applies and substituting the second sum on the right by Eq.\eqref{funnyeq2} we obtain the desired. This completes the proof.
\end{proof}

\begin{rem}\label{rem00}
    We can apply $f_{t^Q}^{q_{t^Q,k}}$ on $R(t^Q+1)$ since $x^{t^Q+1}(i) \ge \sum_{j=1}^{t^Q}q_{j,k}> 0$. 
\end{rem}
\subsection{Action of the operators \texorpdfstring{$E$}{E}} We now define elements $C(j)$ by $$C(t^Q):=f_{t^Q}^{q_{t^Q,k}}(R(t^Q+1)),\ \ C(j):=e_j^{\epsilon_j-q_{j,k}}(C(j-1))\ \text{ for } t^Q<j\le i.$$
We begin with the following proposition.
\begin{prop}
     The following hold.
    \begin{enumerate}
        \item Let $t^Q\le j\le i-2.$ For any $1\le s\le \ell^{j+2},$ we have 
    \begin{equation}\label{eqej}
        S^{C(j)}(j+1,r^{j+2}(s))=-\sum_{p=t^Q}^j q_{p,k}+m-\sum_{p=r^{j+2}(s+1)}^n x^{j+2}(p),
    \end{equation}
    and for $1\le s< \ell^{j+2}$, if $r^{j+2}(s)<r<r^{j+2}(s+1)$, then we have 
$$S^{C(j)}(j+1,r)\le S^{C(j)}(j+1,r^{j+2}(s)).$$
    \item For $t^Q\le j\le i-1,$ the element $C(j)$ is given by
    \begin{equation}\label{form0}
        C(j)_{s,\ell}=\begin{cases}
      r_{s,\ell}  &\ \text{ if } s\le j \text{ and } \ell\neq i,\\
        q_{s,k} &\ \text{ if } s\le j \text{ and } \ell=i,\\
        x^{j+2}(i)-\sum_{p=t^Q}^j r_{p,k}    &\ \text{ if } s=j+1 \text{ and } \ell=i,\\
        r_{s,\ell}+x^{j+2}(\ell)  &\ \text{ if } s=j+1 \text{ and }\ell\neq i,\\
        m-\sum_{p=k+1}^nr_{i,p}-x^i(i) &\ \text{ if } (s,\ell)=(i,i),\\
        r_{s,\ell}+x^{s+1}(\ell)-x^{s}(\ell) &\ \text{ otherwise}.
    \end{cases}
    \end{equation}
    \end{enumerate}
\end{prop}
\begin{proof}
    Let $C(j)$ be of the given form for some $t^Q\le j< i-2.$ Then the $j,j+1$ and $j+2$ columns of $C(j)$ are given by
    \begin{align}\label{form1}
        \begin{array}{|c|c|c|}\hline
       j\text{-th} & (j+1)\text{-th} & (j+2)\text{-th} \\ \hline
        q_{j,k} & x^{j+2}(i) - \sum_{p=t^Q}^j q_{p,k} & x^{j+3}(i) - x^{j+2}(i) \\ \hline
        0 & 0 & 0 \\ \hline
        \vdots & \vdots & \vdots \\ \hline
        0 & 0 & 0 \\ \hline
        q_{j,k+1} & q_{j+1,k+1}+x^{j+2}(k+1) & q_{j+2,k+1}+x^{j+3}(k+1) - x^{j+2}(k+1) \\ \hline
        \vdots & \vdots & \vdots \\ \hline
        q_{j,n} & q_{j+1,n}+x^{j+2}(n) & q_{j+2,n}+x^{j+3}(n) - x^{j+2}(n) \\ \hline
    \end{array}
    \end{align}
    For $2\le s\le \ell^{j+2}$, the quantity $S^{C(j)}(j+1,r^{j+2}(s))$ is equal to 
    \begin{align}\label{form2}
       \nonumber &(x^{j+2}(i)-\sum_{p=t^Q}^j q_{p,k})+\sum_{p=k+1}^{r^{j+2}(s)} (q_{j+1,p}+x^{j+2}(p))+\sum_{p=r^{j+2}(s)}^n (q_{j+2,p}+x^{j+3}(p)-x^{j+2}(p))\\ 
        \nonumber&=-\sum_{p=t^Q}^j q_{p,k}+\sum_{p=k+1}^{r^{j+2}(s)} q_{j+1,p}+\sum_{p=r^{j+2}(s)}^n (q_{j+2,p}+x^{j+3}(p))+\sum_{p=i}^{r^{j+2}(s)}x^{j+2}(p)-\sum_{p=r^{j+2}(s)}^n x^{j+2}(p)\\
        \nonumber&=-\sum_{p=t^Q}^j q_{p,k}+S^{j+2}(r^{j+2}(s))+\sum_{p=i}^{r^{j+2}(s-1)}x^{j+2}(p)-\sum_{p=r^{j+2}(s+1)}^n x^{j+2}(p) \ (\text{using Eq.\eqref{polrs}})\\
        &=-\sum_{p=t^Q}^j q_{p,k}+m-\sum_{p=r^{j+2}(s+1)}^n x^{j+2}(p) \ \ (\text{using Eq.\eqref{recx}})
    \end{align}
    as required. Now for $s=1,$ we have $r^{j+2}(1)=i$ by proposition \ref{propF} and remark \ref{xandr}. Therefore, $S^{C(j)}(j+1,r^{j+2}(1))$ is equal to 
    \begin{align}\label{form3}
       \nonumber =& (x^{j+2}(i)-\sum_{p=t^Q}^j q_{p,k})+(x^{j+3}(i)-x^{j+2}(i))+\sum_{p=k+1}^n (q_{j+2,p}+x^{j+3}(p)-x^{j+2}(p))\\
        \nonumber=&-\sum_{p=t^Q}^j q_{p,k}+ x^{j+3}(i)+ \sum_{p=k+1}^n (q_{j+2,p}+x^{j+3}(p))-\sum_{p=k+1}^nx^{j+2}(p)\\
        \nonumber=&-\sum_{p=t^Q}^j q_{p,k}+S^{j+2}(i)-\sum_{p=k+1}^nx^{j+2}(p)\ \ \text{(using Eq.\eqref{polrs})}\\
        =&-\sum_{p=t^Q}^j q_{p,k}+m-\sum_{p=k+1}^nx^{j+2}(p)\ \ \text{(by lemma \ref{maxm})}.
    \end{align}
    which proves Eq\eqref{eqej} for $C(j).$ 
    \medskip
    
    Now assume that $r$ satisfies the given condition. Let $r\geq k+1$. Using a similar computation as in part (1) in the proof above, we obtain:
    \begin{align*}
        S^{C(j)}(j+1,r)&=-\sum_{p=t^Q}^j q_{p,k}+S^{j+2}(r)+\sum_{p=i}^{r^{j+2}(s)}x^{j+2}(p)-\sum_{p=r^{j+2}(s+1)}^{n}x^{j+2}(p)\\
        &\le-\sum_{p=t^Q}^j q_{p,k}+S^{j+2}(r^{j+2}(s+1))+\sum_{p=i}^{r^{j+2}(s)}x^{j+2}(p)-\sum_{p=r^{j+2}(s+1)}^{n}x^{j+2}(p)\\
        &=-\sum_{p=t^Q}^j q_{p,k}+m-\sum_{p=r^{j+2}(s+1)}^{n}x^{j+2}(p)\quad \ \text{ (by Eq.\eqref{recx})}\\
        &=S^{C(j)}(j+1,r^{j+2}(s))
    \end{align*}
    where the inequality is obtained by definition \ref{defr}. Now if $i<r<k+1$, a similar computation shows that $S^{C(j)}(j+1,r)$ is equal to
    $$S^{C(j)}(j+1,r)=-\sum_{p=t^Q}^j q_{p,k}+ x^{j+2}(i)+ \sum_{p=k+1}^n (q_{j+2,p}+x^{j+3}(p))-\sum_{p=r^{j+2}(2)}^nx^{j+2}(p).$$
    Now using \ref{polrs} and considering the element $f_{j+2}^{x^{j+2}(i)}R(j+3)$, we obtain that $$x^{j+2}(i)+ \sum_{p=k+1}^n (q_{j+2,p}+x^{j+3}(p))\le m,$$ which proves the second statement of part $(1)$ for the fixed $j.$ \medskip

    We now claim that $C(j+1)$ is of the required form. Let $1\le s< \ell^{j+2}.$ Note that by Eq.\eqref{eqej}, we have 
    $$S^{C(j)}(j+1,r^{j+2}(s))+x^{j+2}(r^{j+2}(s+1))=S^{C(j)}(j+1,r^{j+2}(s+1)).$$
    Now if $r^{j+2}(s)<r<r^{j+2}(s+1)$ then by part (1), we have $S^{C(j)}(j+1,r)\le S^{C(j)}(j+1,r^{j+2}(s))$ and thus 
    $$S^{C(j)}(j+1,r)+x^{j+2}(r^{j+2}(s+1))\le S^{C(j)}(j+1,r^{j+2}(s+1)).$$ 
    Now by figure \ref{figelei}, Eq.\eqref{polrs} and lemma \eqref{maxm} we have that (since $C(j)_{j+2,n}=0$) 
    $$\epsilon_{j+1}(C(j))=(m-\sum_{p=t^Q}^j q_{p,k})-(m-\sum_{p=i}^{n}x^{j+2}(p))=\sum_{p=i}^{n}x^{j+2}(p)-\sum_{p=t^Q}^j q_{p,k}\geq q_{j+1,k},$$ where the inequality is obtained by proposition \ref{propF}(a). Therefore, the resulting element $e_{j+1}^{\epsilon_{j+1} - q_{j+1,k}} C(j)$ is obtained by the successive application of $e_{j+1}$ along the rows
$$r^{j+2}(\ell^{s+2}),\ r^{j+2}(\ell^{s+2} - 1),\ \dots,\ r^{j+2}(2),\ r^{j+2}(1),$$
where the operator $e_{j+1}$ acts
$$x^{j+2}(r^{j+2}(\ell^{s+2})),\ x^{j+2}(r^{j+2}(\ell^{s+2} - 1)),\ \dots,\ x^{j+2}(r^{j+2}(2)),\ x^{j+2}(r^{j+2}(1)) - q_{j+1,k}$$
many times respectively. Now the claim follows from  figure \ref{figelei}.

Thus, we have proved that given any $t^Q\le j< i-2,$ and given that $C(j)$ is of the form \eqref{form0}, the conclusions in $(1)$ hold for $C(j).$ Moreover, $C(j+1)$ has the form \eqref{form0}. Now, by \eqref{polrs} and remark \ref{rem00}, we have that the element $C(t^Q)$ is of the from \eqref{form0}. Therefore, by induction on $j$, the proposition holds for all $t^Q\le j< i-2$ and the element $C(i-2)$ is of the form \eqref{form0}. Although $C(i-2)$ is not of the form \eqref{form1} (only differs at $(i,i)$ position), nevertheless, the same computation shows that \eqref{form2} and \eqref{form3} hold even for $C(i-2)$ and moreover $C(i-1)$ has the form \eqref{form0}. This completes the proof.
\end{proof}
\begin{cor}
    $C(i)$ is given by 
    $$C(i)_{s,t}=\begin{cases}
        q_{s,k} & \text{ if }  t=i,\\
        r_{s,t} & \text{ else}.
    \end{cases}
    $$
\end{cor}
\subsection{Action of the operators \texorpdfstring{$H$}{H}}
The following lemma finishes the proof of theorem \ref{thmpolytope}.
\begin{lem}\label{lemH}
    The operator $H$ maps $C(i)$ to $Q.$
\end{lem}
\begin{proof}
    By figure~\ref{figflgri}, it is immediate to see that $f_{i+1}^{\varphi_{i+1}} C(i)$ is the element where the $i$-th row is zero and the $(i+1)$-th row is the same as the $i$-th row of $C(i)$ and all other rows are unchanged. In other words, $f_{i+1}^{\varphi_{i+1}}$ moves the $i$-th row of $C(i)$ to the $(i+1)$-th row. A similar argument shows that the operator $H$ moves the $i$-th row of $C(i)$ to the $k$-th row, setting all rows above the $k$-th row to zero, while keeping all other rows of $C(i)$ unchanged. This is precisely $Q$. This completes the proof.
\end{proof}
Let $A$. Let $k$ be the minimum non-zero row index of $A$. Set $A^{(0)}=0,$ the highest weight element. For $ 1\le j\le n-k+1$, let $A^{(j)}$ be obtained from $A^{(j-1)}$ by replacing the $(n-j+1)$-th row of $A^{(j-1)}$ by the $(n-j+1)$-th row of $A.$  If $T^{(0)}\in \mathrm{SSYT}(m\omega_i)$ is the highest weight tableau, then we have the formula to determine the image of the map $\varphi_{i,m}$ in Eq.\eqref{cryisomap}.
\begin{cor}
    Then the map $\varphi_{i,m}$ in Eq.\eqref{cryisomap} is given by $$\varphi_{i,m}(A)=\prod_{p=1}^{n-k+1}\mathcal{K}^{A^{(n-k-p+2)}} (T^{(0)}).$$
\end{cor}
\section{Explicit isomorphisms for the nodes near the ends of the Dynkin diagram}\label{lowcase}
In this section, we construct explicit crystal isomorphisms between the tableau and polytope realizations of the Kirillov–Reshetikhin crystal $KR^{i,m}$ in the cases $i \leq 2$ and $i \geq n-1$. Unlike in theorem \ref{thmpolytope}, where both raising and lowering operators are involved, here we construct a path in the crystal graph in the polytope realization from the highest weight vector to any vertex using lowering operators alone. 

For convenience of notation, we adopt the following convention which will be used throughout this section: for a square tableau $T$ of shape $(m^i)$, the $k$-th row can contain integers in the interval $[k,\ n-i+k+1]$.  We encode $T$ as a matrix $T_{st}$ of order $i\times (n+2-i)$ with $T_{st}$ denoting the multiplicity of $s+t-1$ in the $s$-th row of $T$. For example, we write the tableau (with $n=5$)
\begin{equation}\label{eq100}
    \begin{array}{|c|c|c|c|c|c|c|}\hline
    1 &1 &2 & 2&2 &4 &4 \\ \hline
    2 &3 &3 &3 &4 &5 &5 \\ \hline
    3&4 &4 &6 &6 &6 &6 \\\hline
\end{array}\qquad  \text{as}\qquad \begin{array}{|c|c|c|c|}
\hline
    2 &3 &0 &2 \\ \hline
    1&3 &1 &2 \\ \hline
    1&2 &0 &4 \\ \hline
\end{array}
\end{equation}
In many situations, it is often convenient to record only the nonzero columns of the matrix, whenever they determine the others. For instance, for large $n$, we write
$$\begin{array}{|c|c|c|c|}\hline
     1&1 &\dots &1 \\ \hline
     2&2 &\dots &2 \\ \hline
\end{array}\ \ \longleftrightarrow\ \ \begin{array}{|c|}\hline
         m  \\ \hline
         m \\\hline
    \end{array}=\ \begin{array}{|c|c|}\hline
         m& 0 \\ \hline
         m & 0\\\hline
    \end{array}
    =\ \begin{array}{|c|c|c|}\hline
         m& 0 & 0 \\ \hline
         m & 0 & 0\\\hline
    \end{array}\ = \ \cdots .$$ 
    Finally, in certain cases, we shall use the symbol $*$ to denote entries that are uniquely determined by the others. For instance, because each row of the matrix in \eqref{eq100} has total sum $7$, the entries in the first column are uniquely determined by this condition and may therefore be denoted by $*$.\medskip

    If $j$ is the largest entry in the tableau $T$, then $j$ must lie in the $i$-th row and the largest entry in the $i-k$-th row is at most $j-k$. Thus we have $T_{st}=0$ for all $t\geq j-i+2.$ We shall consider only the first $j-i+1$ columns. Note that for any $j \in [n]$, the entries $j$ and $j+1$ in the same column occur in consecutive rows, with $j+1$ appearing below. The signature rule in Section~\ref{seccrytab} can be described as follows: in each column that contains $j$ and $j+1$, pair every $j+1$ with the $j$ immediately above it.  Next, any remaining unpaired $j+1$ is paired with an unpaired $j$ in some column strictly to its right, provided no column that lies in between contains an unpaired $j$ or $j+1$. Finally, $f_j$ decreases the matrix entry of the unpaired $j$ with the smallest row index by 1; equivalently, it subtracts $1$ from the matrix entry of the unpaired $j$ in the most north–eastern position and $f_j$ adds $1$ to the matrix entry immediately right to it.
\subsection{The case \texorpdfstring{$i\le 2$}{i≤2}}\label{ssle2} Throughout this section we shall assume that $n$ is large (say $n\geq 4$) and $i\le 2.$ We begin with the following lemma, which associates to each $A \in B^{i,m}$ a path $P(i)^A$ in the crystal graph from the highest weight element, using only lowering crystal operators. Let $A\in B^{i,m}$. Define the operators for $2\le \ell \le n$,
$$X(i)_\ell^A:=\left(\prod_{j=i+1}^\ell f_{\ell+i+1-j}^{\sum_{p=1}^i a_{p,\ell}}\right)\left(\prod_{j=2}^i f_j^{\sum_{p=2}^j a_{p,\ell}}\right),\ \ G(i)^A:=\prod_{p=1}^i f_p^{\sum_{r=i}^n a_{1,r}}.$$

\begin{lem}\label{formule2}
    Let $i\le 2.$ With all the notations as above, the operator $$P(i)^A:=\left(\prod_{\ell=2}^n X(i)^A_\ell\right)G(i)^A$$ maps the highest weight element $A^{(0)}$ to $A$ in $B^{i,m}.$ 
\end{lem}
\begin{proof}
    For $2\le k\le n,$ set $$R(i,k) =\left(\prod_{\ell=k}^n X(i)_\ell^A\right)G(i)^A(A^{(0)}).$$ It suffices to show, by backward induction on $k$, that $R(i,k)\in B^{i,m}$ is given by
 \begin{equation}\label{ind00}
     R(i,k)_{s,t} = \begin{cases}
        \sum_{j=2}^{k-1}a_{1,j} & \text{ if } (s,t)=(1,i),\\
        a_{s,t} & \text{ if } k \le t \le n, \\
        0 & \text{ otherwise}.
    \end{cases}
 \end{equation}
Since $\sum_{j=i}^n a_{1,j} \leq m$, Figure~\ref{figflei} shows that $W := G(i)^A (A^{(0)})$ is well-defined whose only nonzero entry lies in position $(1,i)$, with value $\sum_{j=i}^n a_{1,j}$.\smallskip

 \noindent\textbf{Base case $(k=n)$:} Note that we have $\sum_{j=i}^n a_{1,j}+\sum_{j=2}^i a_{j,n}\le m,$ since it is a sum of entries of a Dyck path. Hence $X_n:=\prod_{j=2}^i f_j^{\sum_{p=2}^j a_{p,n}}(W)$ is well-defined given explicitly by $$(X_n)_{s,t}=\begin{cases}
    a_{s,n} & \text{ if } t=i,\ 1<s\le i ,\\
     W_{s,t} &\text{ else}.
 \end{cases}$$
By the definition of the crystal operators in Section~\ref{secpoly} and argument similar to lemma \ref{lemH}, it follows immediately that $R(i,n)$ has the form in~\eqref{ind00}.\smallskip

\noindent\textbf{Induction Step:} Assume the statement holds for $k+1$ with some $2 < k < n$.  
 Again we have $\sum_{j=i}^{k}a_{1,j}+\sum_{j=2}^i a_{j,k}+\sum_{j=k+1}^n a_{i,j}\le m$, since it is a sum of entries of a sub-Dyck path, the element $X_k:=\prod_{j=2}^i f_j^{\sum_{p=2}^j a_{p,k}}(R(i,k+1))$ is well defined and is given by 
 $$(X_k)_{s,t}=\begin{cases}
    a_{s,t} & \text{ if } t=i,\ 1<s\le i ,\\
     R(i,k+1)_{s,t}&\text{ else}.
 \end{cases}$$
 Once again, by the definition of the crystal operators in Section~\ref{secpoly} and argument similar to lemma \ref{lemH}, we obtain that $R(i,k)$ has the required form~\eqref{ind00}. This completes the proof.
\end{proof}
\begin{rem}
    The formula in lemma \ref{formule2} does not hold if $i\geq 3$. Let $$ A=
        \begin{array}{|c|c|c|}\hline 
            0 & 1 & 1 \\ \hline 
            1 & 3 & 1 \\ \hline 
            1 & 3 & 4 \\ \hline 
        \end{array}.$$

        \noindent Then we have 

        \begin{tikzpicture}[every node/.style={inner sep=4pt, font=\footnotesize}]
    \node (v1) {$\begin{array}{|c|c|c|}\hline 
            0 & 0 & 0 \\ \hline 
            0 & 0 & 0 \\ \hline 
            0 & 0 & 0 \\ \hline 
        \end{array}$};
    \node[right=2cm of v1] (v2) {$\begin{array}{|c|c|c|}\hline 
            2 & 0 & 0 \\ \hline 
            0 & 0 & 0 \\ \hline 
            0 & 0 & 0 \\ \hline 
    \end{array}$};
    \node[right=2cm of v2] (v3) {$\begin{array}{|c|c|c|}\hline 
            2 & 3 & 4 \\ \hline 
            0 & 0 & 0 \\ \hline 
            0 & 0 & 0 \\ \hline 
        \end{array}$};
        \node[right=2cm of v3] (v4) {$ \begin{array}{|c|c|c|}\hline 
            1 & 0 & 0 \\ \hline 
            0 & 0 & 0 \\ \hline 
            1 & 3 & 4 \\ \hline 
        \end{array}$};
        \node[below=1cm of v4] (v5) {$ \begin{array}{|c|c|c|}\hline 
            1 & 1 & 3 \\ \hline 
            0 & 0 & 0 \\ \hline 
            1 & 5 & 2 \\ \hline 
        \end{array}$};
        \node[left=2cm of v5] (v6) {$\begin{array}{|c|c|c|}\hline 
            0 & 0 & 0 \\ \hline 
            1 & 1 & 3 \\ \hline 
            1 & 5 & 2 \\ \hline 
        \end{array}$};
        \node[left=2cm of v6] (v7) {$(A\neq)\ \ \begin{array}{|c|c|c|}\hline 
            0 & 0 & 2 \\ \hline 
            1 & 1 & 3 \\ \hline 
            1 & 6 & 1 \\ \hline
        \end{array}$};
    \draw[->] (v1) -- node[midway, above] {$f_1^2f_2^2f_3^2$} (v2);
    \draw[->] (v2) -- node[midway, above] {$f_2^3f_3^7$} (v3);
    \draw[->] (v3) -- node[midway, above] {$f_5^8f_4^8$} (v4);
    \draw[->] (v4) -- node[midway, right] {$f_2^3f_3^4$} (v5);
    \draw[->] (v5) -- node[midway, above] {$f_4^5$} (v6);
    \draw[->] (v6) -- node[midway, above] {$f_2f_3^2$} (v7);
\end{tikzpicture}
\end{rem}
Thus the map in $\varphi_{i,m}$ in \eqref{cryisomap} is also given by
\begin{equation}\label{mappsii}
    \varphi_{i,m}:B^{i,m}\to \mathrm{SSYT((m^i))},\quad A\mapsto P(i)^A(T^{(0)})
\end{equation}
In the rest of this section, we shall determine the explicit image of this map.
\noindent Given $A\in B^{i,m}$, set $$\Sigma_{n}^A := - a_{2,n},\quad \chi^A_{n}:=1$$ and for $2 \le k < n$ define recursively by downward induction $$\Sigma^A_{k} := a_{1,k+1} + (1-\chi^A_{k+1})\Sigma^A_{k+1} - a_{2,k},\ \text{ where } \chi^A_{k} := \begin{cases}1 & \text{if } \Sigma^A_{k} < 0,\\ 0 & \text{if } \Sigma^A_{k} \ge 0\end{cases}.$$
Note that if $i=1,$ then we have $\chi_k^A\Sigma_k^A=0$ for all $2\le k\le n.$ It is immediate by induction that for any $2\le k\le n$ we have 
\begin{equation}\label{essequa}
    \sum_{j=k}^n (a_{1,j}-a_{2,j})=a_{1,k}+\sum_{j=k}^n \chi_j^A\Sigma_j^A+(1-\chi_k^A)\Sigma_k^A.
\end{equation} 
The following theorem provides the explicit image of $\varphi_{i,m}$ defined in \eqref{mappsii}.
\begin{thm}
Fix $i\le 2$ and $A\in B^{i,m}$. For $2\le k\le n$ write $\Sigma_k:=\Sigma_k^A$ and $\chi_k:=\chi_k^A$. Then, for every $k=2,\dots,n$, the tableau $$T^{(k)}:=X(i)_k^AX(i)_{k+1}^A\cdots X(i)_n^AG(i)^A(T^{(0)})=(T^{(k)}_{st})$$ is given by 
$$T^{(k)}_{st}=\begin{cases}
        m-\sum_{j=i}^na_{1,j}& \text{ if } (s,t)=(i-1,1),\\
        \sum_{j=2}^{k} a_{1,j} + (1-\chi_{k})\Sigma_{k} & \text{ if } (s,t)=(i-1,2),\\
        m - \sum_{j=i}^{n} a_{1,j} + \sum_{j=k}^{n} \chi_j\Sigma_j & \text{ if } (s,t)=(i,1),\\
        \sum_{j=i}^{k-1} a_{1,j} & \text{ if } (s,t)=(i,2)\text{ and }k\geq 3,\\
        a_{2,t-1} + \chi_{t-1}\Sigma_{t-1} & \text{ if } s=i-1 \text{ and } k+1\le t \le n,\\
        a_{1,t+i-2} - \chi_{t+i-2}\Sigma_{t+i-2} & \text{ if } s=i \text{ and } k-i+2\le t\le n-i+2.
    \end{cases}$$
    In particular, the image of $A$ under the map $\Psi_i$ in Eq.\eqref{mappsii}(or \eqref{cryisomap}) is given by the tableau when $k=2.$
\end{thm}
\begin{proof}
A direct computation for $i=1$ case shows that 
$$T^{(k)}_{1,t}=\begin{cases}
    m-\sum_{j=1}^n a_{1,j} & \text{ if } t=1,\\
    \sum_{j=1}^{k-1} a_{1,j} & \text{ if } t=2,\\
    a_{1,t-1} & \text{ if } k+1\le t\le n+1,\\
    0 & \text{ else.}
\end{cases}$$ which proves the theorem for $i=1.$ So we can assume that $i=2.$ Throughout the proof, we use the signature rule described in Section~\ref{seccrytab} without further mention. We proceed by backward induction on $k$. For the base case $k = n$, observe that
    \begin{equation}\label{eq00}
        \begin{array}{|c|}\hline
         m  \\ \hline
         m \\\hline
    \end{array}\ \ 
    \xmapsto{f_1^{\sum_{j=2}^na_{1j}}\circ f_2^{\sum_{j=2}^na_{1j}}}\ \
    \begin{array}{|c|c|}\hline
         m - \sum_{j=2}^na_{1j}& \sum_{j=2}^na_{1j}\\ \hline
         m - \sum_{j=2}^na_{1j}& \sum_{j=2}^na_{1j} \\\hline
    \end{array}\ .
    \end{equation}
    Now, since $\sum_{j=2}^n a_{1j} + a_{2n} \leq m$, we can apply $f_2^{a_{2n}}$ to the tableau on the right-hand side of Equation~\eqref{eq00}, obtaining:
    $$\begin{array}{|c|c|}\hline
         *& \sum_{j=2}^na_{1j}\\ \hline
        *& \sum_{j=2}^na_{1j}+a_{2n} \\\hline
    \end{array}\ .$$
    Now a straightforward computation shows that
    $$\begin{array}{|c|c|}\hline
         *& \sum_{j=2}^na_{1j}\\ \hline
        *& \sum_{j=2}^na_{1j}+a_{2n} \\\hline
    \end{array}\ \xmapsto{\prod_{j=3}^{n} f_{n+3-j}^{a_{1,n}+a_{2,n}}}\ \begin{array}{|c|c|c|c|c|}\hline
        *& \sum_{j=2}^na_{1j} & 0 & \dots& 0\\ \hline
         *& \sum_{j=2}^{n-1}a_{1j} &0 &\dots &a_{1,n}+a_{2,n}\\\hline
    \end{array}\ ,$$ as desired.\medskip

    Assume now that $T^{(k)}$ is of the given form for some $2<k\le n$.
    Define $T^{(k-1,2)}=f_2^{a_{2,k-1}}T^{(k)}$. Since the number of unpaired $2$ in the first row is $a_{1,k}+(1-\chi_{k})\Sigma_{k}$, the total number of unpaired $2$ in $T^{(k)}$ is $T^{(k)}_{2,{\color{blue}1}}+a_{1,k}+(1-\chi_{k})\Sigma_{k}$ and by Eq.\eqref{essequa} we obtain that
     $$m-\sum_{j=2}^n
a_{1,j}+\sum_{j=k}^n \chi_j\Sigma_j+a_{1,k}+(1-\chi_k)\Sigma_k=m-\sum_{j=2}^{k-1} a_{1,j}-\sum_{j=k}^n a_{2,j}\geq a_{2,k-1}.$$
     Therefore $T^{(k-1,2)}$ is a well defined tableau. It is now immediate that 
     $$T^{(k-1,2)}_{st}=\begin{cases}
         T^{(k)}_{s+1,t} +(1-\chi_{k-1})\Sigma_{k-1} & \text{ if } (s,t)=(1,2),\\
         T^{(k)}_{st}+a_{2,k-1}+\chi_{k-1} \Sigma_{k-1} & \text{ if } (s,t)=(1,3),\\
         T^{(k)}_{st}+\chi_{k-1} \Sigma_{k-1} & \text{ if } (s,t)=(2,1),\\
         T^{(k)}_{st}-\chi_{k-1} \Sigma_{k-1} & \text{ if } (s,t)=(2,2),\\
         T^{(k)}_{st} & \text{ else}.
     \end{cases}$$
     Since $k\geq 3$ and $a_{2,k-1}+\chi_{k}\Sigma_{k}\le a_{2,k-1}$ again using signature rule we have that $T^{(k-1,3)}=f_{3}^{a_{1,k-1}+a_{2,k-1}}T^{(k-1,2)}$ is given by 
    $$T^{(k-1,3)}_{st}=\begin{cases}
           T^{(k-1,2)}_{st}+a_{2,k-1}+\chi_{k-1} \Sigma_{k-1} & \text{ if } (s,t)=(1,4),\\
         T^{(k-1,2)}_{st} -  a_{2,k-1} - \chi_{k-1}\Sigma_{k-1} & \text{ if } (s,t)=(1,3),\\
         T^{(k-1,2)}_{st} + a_{1,k-1}-\chi_{k-1} \Sigma_{k-1}& \text{ if } (s,t)=(2,3),\\
         T^{(k-1,2)}_{st} - a_{1,k-1} +\chi_{k-1} \Sigma_{k-1}& \text{ if } (s,t)=(2,2),\\
         T^{(k-1,2)}_{st} & \text{ else}.
    \end{cases}$$
    It is now immediate that the tableau $T^{(k-1)}_{st}=\prod_{j=3}^{k-1} f_{k-1+3-j}^{a_{1,k-1}+a_{2,k-1}}f_2^{a_{2,k-1}} (T^{(k)})$ is given by (noting the fact that $T^{(k)}_{1,3}=0$)
    $$T^{(k-1)}_{s,t}=\begin{cases}
         T^{(k)}_{s+1,t} +(1-\chi_{k-1})\Sigma_{k-1}  & \text{ if } (s,t)=(1,2),\\
         T^{(k)}_{s,t}+a_{2,k-1}+\chi_{k-1} \Sigma_{k-1}  & \text{ if } (s,t)=(1,k),\\
          T^{(k)}_{s,t}+\chi_{k-1} \Sigma_{k-1} & \text{ if } (s,t)=(2,1),\\
         T^{(k)}_{s,t}-a_{1,k-1} & \text{ if } (s,t)=(2,2),\\
         T^{(k)}_{s,t} + a_{1,k-1}-\chi_{k-1} \Sigma_{k-1}  &\text{ if } (s,t)=(2,k-1),\\
         T^{(k)}_{s,t}  &\text{ else.}
    \end{cases}$$
    The proof is now complete by induction.
\end{proof}
Although we shall not use it anywhere, it is interesting to observe that the following holds for $3 \le k \le n$:
\[
\min\{T^{(2)}_{r,k} : 1\le r\le 2\}  = \min\{a_{s,t} \in A : s + t = k + 1\}.
\]
We suspect that this is true in general, i.e., if $A\in B^{i,m}$ and $T=\varphi_{i,m}(A)$ (see Eq.\eqref{cryisomap}) is written in matrix notation as in the beginning of this section, then for $i+1 \le k \le n-i+2$, we have
\[
\min\{T_{r,k} : 1\le r\le i\}  = \min\{a_{s,t} \in A : s + t = k + 1\}.
\]
\subsection{The case \texorpdfstring{$i\geq n-1$}{i≥n−1.}} 
Throughout this subsection, we shall assume that $n$ is large (say $n\geq 4$) and $i$ satisfies $i\geq n-1.$ Since for any $i\in [n]$ the crystals $KR^{i,m}$ and $KR^{n-i+1,m}$ are dual of each other, the formulae for this subsection can be obtained by dualizing the formulae in section \ref{ssle2}. The following lemma is the dual version of the lemma \ref{formule2}. Let $A\in B^{i,m}$ be a arbitrary. Recall the convention of indexing rows and columns of $A.$ We have the following analogous lemma  which we state without proof. For $A\in B^{i,m}$ for $1\le \ell<n$, we define operators
$$Y(i)_\ell^A = \left(\prod_{j=\ell}^{i-1}f_j^{\sum_{p=i}^na_{\ell,p}}\right)\left(\prod_{j=i}^{n-1}f_{n-1+i-j}^{\sum_{p=n-1+i-j}^{n-1}a_{\ell,p}}\right),\ \ H(i)^A = \prod_{j=i}^n f_{n+i-j}^{\sum_{p=1}^i a_{p,n}}$$
\begin{lem}
    Let $i\geq n-1.$ With all the notations as above, the operator $$Q(i)^A:=\left(\prod_{\ell=1}^{n-1} Y(i)_{n-\ell}^A\right)H(i)^A$$ maps the highest weight element $A^{(0)}$ to $A$ in $B^{i,m}.$\qed
\end{lem}
Now we set $\Omega_1^A = -a_{1,n-1}$ and $\nu_1^A = 1$. Define by induction for $2\le k\le n-1$
\begin{equation}\label{defOmega}
    \Omega_k^A = a_{k-1,n} + (1-\nu_{k-1}^A)\Omega_{k-1}^A - a_{k,n-1} \quad \text{and}\quad\nu_k^A = 
\begin{cases}
    1 & \text{if } \Omega_{k}^A < 0,\\ 
    0 & \text{if } \Omega_{k}^A \ge 0.
\end{cases}
\end{equation}
It is immediate by induction that for any $2\le k \le n-1$ we have
\begin{equation}\label{essequa1}
    \sum_{j=1}^{k-1} (a_{j,n}-a_{j,n-1}) = a_{k,n-1} + \sum_{j=1}^{k-1} \nu_j\Omega_j + \Omega_k
\end{equation}
The following is the main result of this subsection which determines the image of an arbitrary element under the map \eqref{cryisomap}. For $1\le k < n,$ define tableau $$T(i)^{(k),A}:=Y(i)_{k}^AY(i)_{k-1}^A\dots Y(i)_1^AH(i)^A(T^{(0)})\in \mathrm{SSYT}(m\omega_i).$$ Note that since $i\geq n-1,$ a tableau in $\mathrm{SSYT}(m\omega_i)$ can contain at most three distinct entries in each row. For $1\le s\le i$, we define the $s$-th row of $T(i)^{(k),A}$ by the tuple $$T(i)^{(k),A}_s:=(T(i)^{(k),A}_{s,1},T(i)^{(k),A}_{s,2},T(i)^{(k),A}_{s,3}).$$
A direct computation for $i=n$ case shows that
    \[
    T^{(k)}_{s,t} = 
    \begin{cases}
         m - \sum_{j=1}^s a_{j,n} &\text{if } t=1 \text{ and } 1\le s\le k,\\
         \sum_{j=1}^s a_{j,n} &\text{if } t=2 \text{ and } 1\le s\le k,\\
         m - \sum_{j=1}^k a_{j,n} &\text{if } t=1 \text{ and } k+1\le s\le n-1,\\
         \sum_{j=1}^k a_{j,n} &\text{if } t=2 \text{ and } k+1\le s\le n-1,\\
         m - \sum_{j=1}^n a_{j,n} &\text{if } (s,t)=(n,1),\\
         \sum_{j=1}^n a_{j,n} &\text{if } (s,t)=(n,2).
    \end{cases}
    \]
\begin{thm}
Let $i=n-1$ and let $A\in B^{i,m}$. Set $\Omega_k^A=\Omega_k,\ \nu_k^A=\nu_k$ for $1\le k\le n-1$ and $T^{(k)}=T(i)^{(k),A}$ for $1\le k< n.$ Then 
\begin{itemize}
    \item for $1\le s\le k-1$, the tuple $T^{(k)}_s$ is given by
    \[
    \left(m - \sum_{j=1}^s (a_{j,n} - \nu_j\Omega_j), -\sum_{j=1}^s\nu_j\Omega_j + (1-\nu_{s+1})\Omega_{s+1}, \sum_{j=2}^{s+1} (a_{j,n-1} + \nu_j\Omega_j)\right).
    \]
    \item for $k\le s\le i-1$, the tuple $T^{(k)}_s$ is given by
    \[
    \left(m - \sum_{j=1}^k (a_{j,n} - \nu_j\Omega_j), -\sum_{j=1}^k \nu_j\Omega_j + a_{k,n} + (1-\nu_{k})\Omega_{k}, \sum_{j=2}^{k} (a_{j,n-1} + \nu_j\Omega_j)\right).
    \]
    \item for $s=i$, the tuple $T^{(k)}_s$ is given by
    $$\left(m - \sum_{j=1}^{i} a_{j,n} + \sum_{j=1}^k \nu_j\Omega_j, -\sum_{j=1}^k\nu_j\Omega_j, \sum_{j=1}^{i} a_{j,n}\right).$$
\end{itemize}
In particular, the image of $A$ under the map $\Psi_i$ in Eq.\eqref{cryisomap} is given by the tableau when $k=n-1.$
\end{thm}
\begin{proof}
    Throughout the proof, we use the matrix notation for the tableau and the signature rule  described in beginning of this section without further mention. We proceed by induction on $k.$ Since $\sum_{p=1}^{n-1} a_{p,n}\le m,$ by applying the signature rule, it is immediate that $T^{(1,n)}:=f_{n}^{\sum_{p=1}^{n-1}a_{p,n}}f_{n-1}^{\sum_{p=1}^{n-1}a_{p,n}}(T^{(0)})$ is given by
    $$T^{(1,n)}_{s,t}=\begin{cases}
        m & \text{ if } 1\le s\le n-2,\ t=1,\\
        m- \sum_{p=1}^{n-1}a_{p,n}& \text{ if } (s,t)=(n-1,1),\\
        \sum_{p=1}^{n-1}a_{p,n}& \text{ if } (s,t)=(n-1,3),\\
        0 & \text{ else}.
    \end{cases}$$
    Now since $\sum_{p=1}^{n-1}a_{p,n}+ a_{1,n-1}$ is the sum of entries along a Dyck path, we have  $m- \sum_{p=1}^{n-1}a_{p,n}\geq a_{1,n-1}$. It follows that $T^{(1,n-1)}:=f_{n-1}^{a_{1,n-1}}T^{(1,n)}$ is a well defined tableau given by:
    $$T^{(1,n-1)}_{s,t}=\begin{cases}
        T^{(1,n)}_{s,t}-a_{1,n-1} & \text{ if } (s,t)=(n-1,1),\\
        T^{(1,n)}_{s,t}+a_{1,n-1} & \text{ if } (s,t)=(n-1,2),\\
        T^{(1,n)}_{s,t} & \text{ else}.
    \end{cases}$$
    Now note that the number of unpaired $n-2$ in the $(n-2)$-th row is $\sum_{p=1}^{n-1} a_{p.n}+a_{1,n-1}$ Thus $f_{n-1}^{a_{1,n-1}+a_{1,n}}(T^{(1,n-1)}).$  Now it is immediate that the tableau $$T^{(1)}=f_1^{a_{1,n-1}+a_{1,n}}\dots f_{n-2}^{a_{1,n-1}+a_{1,n}}(T^{(1,n-1)})$$ is of the given form. This completes the proof of the base step.

    Assume now that $T^{(k)}$ is of the given form for some $1\le k < n-1$. Define $T^{(k+1,n-1)}=f_{n-1}^{a_{k+1,n-1}}T^{(k)}$. Since each $n$ in the $(n-2)$-th row is paired with each $(n-1)$ in the $(n-3)$-th row, the number of unpaired $n-1$ is the $(n-2)$-th row is $a_{k,n} + (1-\nu_k)\Omega_k$. Thus the total number of unpaired $n-1$ in $T^{(k)}$ is $T^{(k)}_{n-1,1} + a_{k,n} + (1-\nu_k)\Omega_k$. Using Eq.\eqref{essequa1} we obtain that
    \[
     m - \sum_{j=1}^{n-1} a_{j,n} + \sum_{j=1}^k \nu_j\Omega_j + a_{k,n} + (1-\nu_k)\Omega_k \ge a_{k+1,n-1}.
    \]
    Therefore $T^{(k+1,n-1)}$ is a well-defined tableau. By the description of the action of $f_{n-1}$ in the beginning of this section, we obtain that
    \[
    T^{(k+1,n-1)}_{s,t}=\begin{cases}
        T^{(k)}_{s,t} + \nu_{k+1}\Omega_{k+1} & \text{if } (s,t) = (n-1,1),\\
        T^{(k)}_{s,t} - \nu_{k+1}\Omega_{k+1} & \text{if } (s,t) = (n-1,2),\\
        T^{(k)}_{s+1,t} + (1-\nu_{k+1})\Omega_{k+1} & \text{if } (s,t) = (n-2,2),\\
        T^{(k)}_{s,t} + a_{k+1,n-1} + \nu_{k+1}\Omega_{k+1} & \text{if } (s,t) = (n-2,3),\\
        T^{(k)}_{s,t} & \text{else}.
    \end{cases}
    \]
    Note that the number of unpaired $(n-2)$ in the $(n-3)$-th and $(n-2)$-th row is $$a_{k,n}+(1-\nu_k)\Omega_k-(1-\nu_{k+1})\Omega_{k+1}=a_{k+1,n-1}+\nu_{k+1}\Omega_{k+1}\geq 0 (\text{ by } Eq.\eqref{defOmega})$$
    and $$\sum_{j=k+1}^{n-1} a_{j,n}-\nu_{k+1}\Omega_{k+1}(\geq a_{k+1,n}-\nu_{k+1}\Omega_{k+1})$$ respectively. Thus $f_{n-2}^{a_{k+1,n-1}+a_{k+1,n}}$ acts on row $(n-3)$ and $(n-2)$-th row $a_{k+1,n-1}+\nu_{k+1}\Omega_{k+1}$ and $a_{k+1,n}-\nu_{k+1}\Omega_{k+1}$ may times respectively. Therefore, the tableau obtained $T^{(k+1,n-2)} = f_{n-2}^{a_{k+1,n-1}+a_{k+1,n}}T^{(k+1,n-1)}$ is given by 
    \[
    T^{(k+1,n-1)}_{s,t}=\begin{cases}
        T^{(k+1,n-1)}_{s,t} - a_{k+1,n} + \nu_{k+1}\Omega_{k+1} & \text{if } (s,t) = (n-2,1),\\
        T^{(k+1,n-1)}_{s,t} + a_{k+1,n} - \nu_{k+1}\Omega_{k+1} & \text{if } (s,t) = (n-2,2),\\
        T^{(k+1,n-1)}_{s+1,t} & \text{if } (s,t) = (n-3,2),\\
        T^{(k+1,n-1)}_{s,t} + a_{k+1,n-1} + \nu_{k+1}\Omega_{k+1} & \text{if } (s,t) = (n-3,3),\\
        T^{(k+1,n-1)}_{s,t} & \text{else}.
    \end{cases}
    \]    
    Note that $(n-3)$-th and $(n-4)$-th row of the tableau $T^{(k+1,n-2)}$  are precisely the $(n-2)$-th and $(n-3)$-th rows of $T^{(k+1,n-1)}$ respectively. Moreover, the total number of unpaired $(n-3)$ in $T^{(k+1,n-2)}$ is equal to $a_{k+1,n-1}+a_{k+1,n}.$ Therefore, a similar argument shows that the tableau $T^{(k+1)}_{s,t} = \prod_{j=k+1}^{n-2}f_j^{a_{k+1,n-1}+f_{k+1,n}}f_{n-1}^{a_{k+1,n-1}}(T^{(k})$ is given by 
    \[
    T^{(k+1)}_{s,t}=\begin{cases}
        T^{(k)}_{s,t} - a_{k+1,n} + \nu_{k+1}\Omega_{k+1} & \text{if } t=1 \text{ and } k+1\le s\le n-2,\\
        T^{(k)}_{s,t} + \nu_{k+1}\Omega_{k+1} & \text{if } (s,t) = (n-1,1),\\
        T^{(k)}_{s,t} - a_{k+1,n-1} - \nu_{k+1}\Omega_{k+1} & \text{if } (s,t) = (k,2),\\
        T^{(k)}_{s,t} - a_{k+1,n-1} + a_{k+1,n} - 2\nu_{k+1}\Omega_{k+1} & \text{if } t=2 \text{ and } k+1\le s\le n-2,\\
        T^{(k)}_{s,t} - \nu_{k+1}\Omega_{k+1} & \text{if } (s,t) = (n-1,2),\\
        T^{(k)}_{s,t} + a_{k+1,n-1} + \nu_{k+1}\Omega_{k+1} & \text{if } t=3 \text{ and } k\le s\le n-2,\\
        T^{(k)}_{s,t} & \text{else}.
    \end{cases}
    \]
    Now the proof is completed by induction.
\end{proof}
\section*{Declaration}

\noindent\textbf{Acknowledgment:} The authors thank Prof. Sankaran Viswanath and Prof. Deniz Kus for useful discussions. The first and the second author acknowledge IIT Madras and IMSc respectively for their financial support and exceptional facilities where the research has been carried out.\vspace{0,2cm} 

\noindent \textbf{Ethical Approval:} Not applicable.\vspace{0,2cm}

\noindent\textbf{Data Availability:} No datasets were generated or analysed during the current study.\vspace{0,2cm}

\noindent\textbf{Competing interest:} The authors declare no competing interests. \vspace{0,2cm}

\noindent\textbf{Contributions:} All authors equally contributed.

\bibliographystyle{plain}
\bibliography{bibliography}
	
\end{document}